\documentclass[12pt,a4paper,oneside]{amsart}

\usepackage[margin=1.2in,marginpar=1in]{geometry}

\usepackage{amsmath,amsthm,amssymb,amsfonts}
\usepackage[dvipsnames]{xcolor}
\usepackage{graphicx}
\usepackage{tikz}
\usepackage[colorlinks=true,linkcolor=NavyBlue,citecolor=ForestGreen]{hyperref}
\usepackage{enumerate}
\usepackage{booktabs}
\usepackage[addtotoc]{abstract}
\usepackage[titletoc,title]{appendix}
\usepackage{mathrsfs}
\usepackage{bm}
\usepackage{float}
\usepackage{wrapfig}
\usepackage[nameinlink]{cleveref}

\newtheorem{theorem}{Theorem}[section]
\newtheorem{lemma}[theorem]{Lemma}

\newtheorem{remark}[theorem]{Remark}

\newtheorem*{remark*}{Remark}

\renewenvironment{proof}[1][\proofname]{\medskip \noindent {\bfseries #1. }}{\hfill \qedsymbol \medskip}

\numberwithin{equation}{section}
\numberwithin{figure}{section}
\numberwithin{table}{section}

\usepackage{tikz-cd}

\usepackage{caption}
\usepackage{subcaption}

\newcommand{\bfi}{\mathbf{i}}
\newcommand{\rmd}{\mathrm{d}}

\DeclareRobustCommand{\SkipTocEntry}[5]{}

\newcommand{\mR}{\mathbb{R}}   
\newcommand{\mC}{\mathbb{C}}   
\newcommand{\mS}{\mathbb{S}}   
\newcommand{\abs}[1]{\lvert #1 \rvert}  
\newcommand{\norm}[1]{\lVert #1 \rVert}  
\newcommand{\br}[1]{\langle #1 \rangle}  

\newcommand{\ol}[1]{\overline{#1}}

\newcommand{\ehat}{\,\widehat{\rule{0pt}{6pt}}\,}

\newcommand{\eps}{\epsilon}

\newcommand{\p}{\partial}

\DeclareMathOperator{\supp}{supp}

\usepackage{comment}

\usepackage{orcidlink}

\begin{document}

\title[Increasing resolution/instability for linear inverse problems]{Increasing resolution and instability for linear inverse scattering problems}

\author[P.-Z. Kow]{Pu-Zhao Kow\,\orcidlink{0000-0002-2990-3591}}
\address{Department of Mathematical Sciences, National Chengchi University, No. 64, Sec. 2, ZhiNan Rd., Wenshan District, 116302 Taipei, Taiwan}
\email{pzkow@g.nccu.edu.tw}

\author[M. Salo]{Mikko Salo\,\orcidlink{0000-0002-3681-6779}}
\address{Department of Mathematics and Statistics, P.O. Box 35 (MaD), FI-40014 University of Jyv\"{a}skyl\"{a}, Finland}
\email{mikko.j.salo@jyu.fi}

\author[S. Zou]{Sen Zou}
\address{School of Mathematical Science, No.220 Road Handan, Shanghai, Fudan University, 200433, China}
\email{szou18@fudan.edu.cn}


\maketitle

\begin{abstract}
In this work we study the increasing resolution of linear inverse scattering problems at a large fixed frequency. We consider the problem of recovering the density of a Herglotz wave function, and the linearized inverse scattering problem for a potential. It is shown that the number of features that can be stably recovered (stable region) becomes larger as the frequency increases, whereas one has strong instability for the rest of the features (unstable region). To show this rigorously, we prove that the singular values of the forward operator stay roughly constant in the stable region and decay exponentially in the unstable region. The arguments are based on structural properties of the problems and they involve the Courant min-max principle for singular values, quantitative Agmon-H\"ormander estimates, and a Schwartz kernel computation based on the coarea formula.
\end{abstract}

\tableofcontents

\begin{sloppypar}

\section{Introduction}

Ill-posedness, or instability, is a central feature of many inverse problems. The associated high sensitivity to noise needs to be taken into account in the design of reliable reconstruction algorithms. We refer the reader to \cite{KRS21InstabilityMechanism} for a discussion of instability in various inverse problems, and to \cite{EKN89TikhonovRegularization} for an overview of regularization theory that addresses this issue. 

In some cases, inverse problems involve parameters that may have an effect on their stability properties. A rigorous analysis of this \emph{increasing stability} phenomenon was initiated by Victor Isakov in the case of unique continuation for the Helmholtz equation \cite{HI04IncreasingStabilityHelmholtz}. There are many subsequent works on increasing stability for inverse problems \cite{Isa11increasingstability,IN12EnergyRegularitymultidimensionGelfand,INUW14increasingstability,ILW16IncreasingStability,IW21IncreasingStabilitySource,KW22RefinedInstability,NUW13IncreasingStabilityHelmholtz,Santacesaria13IncreasingStability,Santacesaria15IncreasingStability2D}. 
The increasing stability phenomenon for the Helmholtz equation was also studied using the Bayesian approach in \cite{KW24IncreasingStabilityBayesian}.

In particular, \cite{Isa11increasingstability} considers the inverse problem of determining a potential $q$ from the Dirichlet-to-Neumann map $\Lambda_{q}$ for the Helmholtz equation $(\Delta + \kappa^{2} + q)u=0$, where $\kappa>0$ is a large frequency, and proves a conditional stability estimate of the form 
\begin{equation}
\norm{q_{1}-q_{2}}_{L^{2}} \le \omega_{\text{\rm H{\"o}l}}(\norm{\Lambda_{q_{1}}-\Lambda_{q_{2}}}) + \omega_{\rm Log}(\norm{\Lambda_{q_{1}}-\Lambda_{q_{2}}}). \label{eq:stability-estimate}
\end{equation}
Here $\omega_{\text{\rm H{\"o}l}}$ and $\omega_{\rm Log}$ are H\"older and logarithmic moduli of continuity, respectively, and the point is that the logarithmic part formally goes to zero as $\kappa\rightarrow\infty$. 

In this article we take an alternative point of view to the increasing stability phenomenon: there may be a number of features (e.g.\ Fourier modes) that can be reconstructed in a stable way, and the number of these stable features increases with the parameter. This could be called \emph{increasing resolution} instead of increasing stability. It turns out that stability estimates of the form \eqref{eq:stability-estimate} may follow from such an increasing resolution analysis. For linear inverse source problems such ideas appear already in \cite{BaoLinTriki2010} and subsequent works.

We give a rigorous study of the increasing resolution phenomenon for two model inverse problems related to scattering phenomena. These problems are addressed in $\mathbb{R}^n$, which avoids the issues with Dirichlet eigenvalues on bounded domains. Moreover, we consider linear inverse problems to reduce matters to singular value estimates. In each problem the measurement is related to some Fourier transform in a ball whose radius grows with the frequency $\kappa$, which indicates that the resolution should increase.

In this article we quantify the increasing resolution phenomenon precisely in terms of the asymptotics of singular values of the forward operator. For the truncated Fourier transform itself  singular value estimates are classical (see \cite{Bonami} and references therein). For related inverse source problems such estimates are given in \cite{BaoLinTriki2010, GriesmaierHankeSylvester2014, GriesmaierSylvester2017, GriesmaierSylvester2017_3D,GriesmaierSylvester2018_Maxwell_elasticity, Mirza1, Mirza2}. The main point of this work is to introduce methods that do not rely on explicit Bessel function estimates, but rather employ structural properties of the inverse problem such as stability and smoothing estimates to study singular value asymptotics.  We note that a similar approach to quantify increasing resolution could be applied to any linear inverse problem, regardless of whether the measurement is related to some Fourier transform or not. 

Our results show that there is a \emph{stable region} where the singular values stay roughly constant as $\kappa\rightarrow\infty$, and an \emph{unstable region} where they decay exponentially. We give optimal estimates for the size of the stable region, and show that the stable region grows (i.e.\ the number of stable features increases) as the frequency $\kappa$ increases. For singular values in the unstable region we show decay rates that lead to instability results as in \cite{KRS21InstabilityMechanism}. 

\addtocontents{toc}{\SkipTocEntry}
\subsection{Unique continuation for Herglotz wave functions} 

As a warmup problem we first investigate the inverse problem of recovering the density of a Herglotz wave function from its values in the unit ball. 
The forward operator will be a scaled version of the adjoint of the operator $\mathcal{F}_{B_R(0)}$ studied in \cite{GriesmaierSylvester2017, GriesmaierSylvester2017_3D}. 
Thus the Bessel function estimates in those works yield singular value estimates as in \Cref{thm:1} at least for $n=2,3$. See also \cite{GriesmaierSylvester2018_Maxwell_elasticity} for similar results for electromagnetic and elastic waves. However, the methods for proving \Cref{thm:1} are based on structural properties of the problem instead of explicit Bessel function estimates, and hence they will be applicable to other inverse problems such as the linearized inverse scattering problem for a potential in \Cref{thm:2}.

Given any $f\in L^{2}(\mS^{n-1})$, the corresponding Herglotz wave function is defined by 
\begin{equation*}
(P_{\kappa}f)(x) := \int_{\mS^{n-1}} e^{\bfi \kappa \omega \cdot x}f(\omega) \, \rmd S(\omega) = (f\,\rmd S)\ehat (-\kappa x) \quad \text{for all $x\in\mR^{n}$.} 
\end{equation*}
This function solves $(\Delta + \kappa^{2})(P_{\kappa}f)=0$ in $\mR^{n}$. Now define the linear map 
\begin{equation}
A_{\kappa} : L^{2}(\mS^{n-1}) \rightarrow L^{2}(B_{1}) ,\quad A_{\kappa}(f) := \kappa^{\frac{n-1}{2}} \left. P_{\kappa}f \right|_{B_{1}}, \label{eq:linear-injective-compact-map1}
\end{equation}
The normalizing constant $\kappa^{\frac{n-1}{2}}$ in the definition of $A_{\kappa}$ will simplify the statements below. By the Agmon-H\"ormander estimate \cite[Theorem~2.1]{AgmonHormander} we have 
\begin{equation}
\norm{A_{\kappa}f}_{L^{2}(B_{1})} \le C(n) \norm{f}_{L^{2}(\mS^{n-1})}. \label{eq:normalized-AG-inequality}
\end{equation}
This shows that $A_{\kappa}$ is bounded. Elliptic regularity shows that $A_{\kappa}$ maps $L^{2}(\mS^{n-1})$ to $H^{s}(B_{1})$ for any $s>0$, and hence \eqref{eq:linear-injective-compact-map1} is compact by compact Sobolev embedding. Moreover, the analyticity of $(f\,\rmd S)\ehat$, or alternatively the unique continuation principle applied to the solution $P_{\kappa}f$, implies that $f$ is uniquely determined by $A_{\kappa}f$. Thus \eqref{eq:linear-injective-compact-map1} is injective and it has a sequence of singular values $\sigma_{j} = \sigma_{j}(A_{\kappa})$ with $\sigma_{1} \ge \sigma_{2} \ge \cdots \rightarrow 0$. 

We are interested in the behavior of the singular values $\sigma_{j}$ of \eqref{eq:linear-injective-compact-map1} with explicit depedence on the key parameter $\kappa$. We will prove that the singular values are roughly constant in the region $j \lesssim \kappa^{n-1}$, called the \emph{stable region}, whereas they decay exponentially in the \emph{unstable region} where $j \gtrsim \kappa^{n-1}$. Here and below, we write $A \lesssim B$ (resp.\ $A \gtrsim B$ or $A \sim B$) for $A \leq CB$ (resp.\ $A \geq C^{-1}B$ or $C^{-1} A \leq B \leq C A$) where $C$ is a constant independent of asymptotic parameters (here $j$ and $\kappa$).

\begin{theorem}\label{thm:1}
Let $n\ge 2$ and $\kappa \ge 1$. The singular values $\sigma_{j}(A_{\kappa})$ of \eqref{eq:linear-injective-compact-map1} satisfy 
\begin{subequations}
\begin{align}
& \sigma_{j}(A_{\kappa}) \sim 1, && \text{for all} \quad j \lesssim \kappa^{n-1}, \label{eq:stable-region-Herglotz} \\
& \sigma_{j}(A_{\kappa}) \lesssim \exp \left( -c \kappa^{-1} j^{\frac{1}{n-1}} \right), && \text{for all} \quad j \gtrsim \kappa^{n-1}, \label{eq:unstable-region-Herglotz}
\end{align}
\end{subequations}
where the constant $c>0$ and the implied constants are independent of $\kappa$ and $j$. 
\end{theorem}

See also \Cref{sec:numerical} for related numerical experiments. We note that the bounds \eqref{eq:stable-region-Herglotz} and \eqref{eq:unstable-region-Herglotz} match when $j \sim \kappa^{n-1}$, which shows that the size of the stable region is optimal. Moreover, by \Cref{rem:Jacobi-Anger} the spherical harmonics are a singular value basis of $A_{\kappa}$. This shows that for $A_{\kappa}$ the stable features are precisely the first $\sim \kappa^{n-1}$ Fourier modes. We also remark that if one removes the normalizing constant $\kappa^{\frac{n-1}{2}}$ from \eqref{eq:linear-injective-compact-map1}, then the singular values in the stable region are $\sim \kappa^{-\frac{n-1}{2}}$. For $\kappa$ large this smallness might affect the quality of reconstructions.

One could also ask for more precise estimates in the ``plunge region'' $j \sim \kappa^{n-1}$ (as in \cite{Bonami}), or for a complementary lower bound for the singular values when $j \gtrsim \kappa^{n-1}$. The latter would correspond to a quantitative unique continuation result for solutions of $(\Delta+\kappa^{2})u=0$. Such results exist, see e.g.\ \cite{ GarciaFerreroRulandZaton,John60UCP}, but we do not consider this point further.  

From \Cref{thm:1} we see that the size of the stable region (i.e.\ the number of features that can be stably recovered) increases as $\kappa$ increases, but the recovery of high frequencies will always be unstable. Consequently, by refining the results in \cite{KRS21InstabilityMechanism}, we are able to derive the following theorem concerning the optimality of increasing stability/resolution of the inverse problem of recovering $f$ from the knowledge of $A_{\kappa}f$, with respect to $\kappa$. 

\begin{theorem}\label{thm:1-instability}
Suppose that all assumptions in \Cref{thm:1} hold. If there exists a non-decreasing function $t\in\mR_{+} \mapsto \omega(t) \in \mR_{+}$ such that 
\begin{equation*}
\norm{f}_{L^{2}(\mS^{n-1})} \le \omega\left(\norm{A_{\kappa}f}_{L^{2}(B_{1})}\right) \quad \text{whenever $\norm{f}_{H^{1}(\mS^{n-1})} \le 1$,}
\end{equation*}
then 
\begin{equation*}
\omega(t) \gtrsim \max \left\{ t , \kappa^{-1}(1 + \log(1/t))^{-1} \right\} \quad \text{for all $0 < t \lesssim 1$,} 
\end{equation*}
where the implied constants are independent of $\kappa$ and $t$. 
\end{theorem}

\addtocontents{toc}{\SkipTocEntry}
\subsection{Linearized inverse scattering}

We move on to the second inverse problem studied in this article. The following facts may be found e.g.\ in \cite{Mel95GeometricScattering,PaivarintaSaloUhlmann2010}. Let $\kappa>0$ be a fixed frequency and let $q\in C_{c}^{\infty}(\mR^{n})$ be an unknown scattering potential. We suppose that we probe the medium by sending an incoming Herglotz wave $u^{\rm inc}=P_{\kappa}f$ where $f\in L^{2}(\mS^{n-1})$. This induces a unique total wave $u^{\rm tot}=u^{\rm inc}+u^{\rm sc}$ solving 
\begin{equation*}
(\Delta+\kappa^{2}+q)u^{\rm tot}=0 \quad \text{in $\mR^{n}$}
\end{equation*}
where $u^{\rm sc}$ satisfies the outgoing Sommerfeld radiation condition 
\begin{equation*}
\lim_{r\rightarrow\infty} r^{\frac{n-1}{2}} (\partial_{r}u^{\rm sc}-\bfi\kappa u^{\rm sc})(r\theta)=0, \quad \text{uniformly for $\theta\in\mS^{n-1}$.}
\end{equation*}
We write $P_{\kappa}(q)f := u^{\rm tot}$ and call $P_{\kappa}(q)$ the \emph{Poisson operator} that maps a boundary data $f$ on the sphere at infinity to the corresponding solution of $(\Delta+\kappa^{2}+q)u=0$ in $\mR^{n}$. 

The scattering measurements at frequency $\kappa$ are encoded by the \emph{scattering matrix} $S_{\kappa}(q)$. This is an operator on $L^{2}(\mS^{n-1})$ that may be defined via 
\begin{equation*}
(S_{\kappa}(q)f)(\theta) = \lim_{r\rightarrow\infty} r^{\frac{n-1}{2}} e^{-\bfi\kappa r} (\partial_{r} u^{\rm tot} + \bfi\kappa u^{\rm tot})(r\theta).
\end{equation*}
After multiplying $S_{\kappa}(q)$ by a suitable normalizing constant, it becomes a unitary operator on $L^{2}(\mS^{n-1})$ that satisfies the integral identity (``boundary pairing'')
\begin{equation}
\left(((S_{\kappa}(q)-S_{\kappa}(0))f,g\right)_{L^{2}(\mS^{n-1})} = c_{n,\kappa} \left((q-0)P_{\kappa}(q)f,P_{\kappa}(0)g\right)_{L^{2}(\mR^{n})}. \label{eq:scattering-operator-def}
\end{equation}
This is an analogue of the Alessandrini identity appearing in inverse boundary value problems. Thus $S_{\kappa}(q)$ can be thought of as an analogue for $\mR^{n}$ of the Dirichlet-to-Neumann map (or more precisely impedance-to-impedance map) for a bounded domain.  Knowing $S_{\kappa}(q)$ is equivalent to knowing the far-field operator, or scattering amplitude, for the equation $(\Delta+\kappa^{2}+q)u=0$ in $\mR^{n}$. 

The linearization of $S_{\kappa}$ at $q=0$ is readily obtained from the identity \eqref{eq:scattering-operator-def}. This linearization, denoted by $F_{\kappa}$ (after multiplying by a suitable constant), is given by the formula 
\begin{equation} \label{fkappa_bilinear}
(F_{\kappa}(h)f,g)_{L^{2}(\mS^{n-1})} = \kappa^{\frac{n-1}{2}} (h P_{\kappa}f,P_{\kappa}g)_{L^{2}(B_{1})}.
\end{equation}
We will show that if $h$ is supported in $\ol{B}_1$, then $F_{\kappa}(h)$ is a Hilbert-Schmidt operator on $L^2(\mS^{n-1})$, and in the case $n=3$ one further has 
\begin{equation} \label{fkappa_norm_first}
\norm{F_{\kappa}(h)}_{\rm HS}^2 \sim \int_{B_{2\kappa}} \frac{\abs{\hat{h}(\xi)}^2}{\abs{\xi}} \,\rmd\xi.
\end{equation}
This suggests that it is natural to consider $F_{\kappa}$ as a compact operator  
\begin{equation}
F_{\kappa} : H_{\overline{B}_{1}}^{-1/2} \rightarrow {\rm HS}\,(L^{2}(\mS^{n-1})), \label{eq:operator-Fk}
\end{equation}
where $H^s_K = \{ f \in H^s(\mR^n) \,:\, \mathrm{supp}(f) \subset K \}$. For simplicity we use $\|\cdot\|_{\rm HS}$ to denote the Hilbert-Schmidt norm on $L^2(\mathbb{S}^{n-1})$. The injectivity of $F_{\kappa}$ follows from \eqref{fkappa_norm_first} and the analyticity of $\hat{h}$. Alternatively, one can prove injectivity of $F_{\kappa}$ by \eqref{fkappa_bilinear} and the Runge approximation fact that $P_{\kappa} f|_{B_1}$ can be used to approximate complex geometrical optics solutions $e^{\rho\cdot x}|_{B_1}$ for suitable $\rho \in \mC^n$ with $\rho\cdot\rho=-\kappa^{2}$. Hence \eqref{eq:operator-Fk} has a sequence of positive singular values $\sigma_{j} = \sigma_{j}(F_{\kappa})$ with $\sigma_{1}\ge \sigma_{2} \ge \cdots \rightarrow 0$. 

As before, we are interested in the behavior of the singular values $\sigma_{j}$ of \eqref{eq:operator-Fk} with explicit dependence on the key parameter $\kappa$. We will show that the singular values are roughly constant in the stable region $j \lesssim \kappa^{n}$, whereas they decay exponentially in the unstable region where $j \gtrsim \kappa^{n}$.  Again, related numerical experiments are given in \Cref{sec:numerical}.

\begin{theorem}[See also \Cref{thm_scattering_refined}] \label{thm:2}
Let $n\ge 2$ and $\kappa \ge 1$. The singular values $\sigma_{j}(F_{\kappa})$ of \eqref{eq:operator-Fk} satisfy 
\begin{subequations}
\begin{align}
& \sigma_{j}(F_{\kappa}) \sim 1, && \text{for all} \quad j\lesssim \kappa^{n}, \label{eq:stable-region-farfield} \\
& \sigma_{j}(F_{\kappa}) \lesssim \kappa^{\alpha(n)} j^{\frac{1}{2n}} \exp\left(-c\kappa^{-\frac{1}{2}}j^{\frac{1}{2n}}\right), && \text{for all} \quad j \gtrsim \kappa^{n}, \label{eq:unstable-region-farfield}
\end{align}
\end{subequations}
where the constant $c>0$ and the implied constants are independent of $\kappa$ and $j$, $\alpha(n)=0$ for $n \geq 3$, and $\alpha(2) = \frac{1}{2}$. 
\end{theorem}

\begin{remark}\label{rem:thm2}
A similar argument, given in \Cref{thm_scattering_refined}, shows that the singular values $\sigma_{j}(F_{\kappa} : L^{2}_{\ol{B}_{1}} \rightarrow {\rm HS})$ satisfy 
\begin{subequations}
\begin{align}
& j^{-\frac{1}{2n}} \lesssim \sigma_{j}(F_{\kappa} : L^{2}_{\ol{B}_{1}} \rightarrow {\rm HS}) \lesssim 1, && \text{for all} \quad j \lesssim \kappa^{n}, \label{eq:stable-region-farfield-L2}\\
& \sigma_{j}(F_{\kappa} : L^{2}_{\ol{B}_{1}} \rightarrow {\rm HS}) \lesssim \kappa^{\alpha(n)} \exp\left(-c\kappa^{-\frac{1}{2}}j^{\frac{1}{2n}}\right), && \text{for all} \quad j \gtrsim \kappa^{n}.\label{eq:unstable-region-farfield-L2} 
\end{align}    
\end{subequations}
The decay estimate \eqref{eq:unstable-region-farfield} is probably not optimal but it is strong enough to show that for $j \sim \kappa^n$ one has $\exp\left(-c\kappa^{-\frac{1}{2}}j^{\frac{1}{2n}}\right) \sim 1$, showing that the size of the stable region is (nearly) optimal. 
\end{remark}

In analogy with \Cref{thm:1-instability}, we are able to derive the following theorem concerning the optimality of increasing stability of the inverse problem of recovering $h$ from the knowledge of $F_{\kappa}(h)$.

\begin{theorem}\label{thm:2-instability}
Suppose that all assumptions in \Cref{thm:2} hold. If there exists a non-decreasing function $t:\mR_{+}\rightarrow\omega(t)\in\mR_{+}$ such that 
\begin{equation*}
\norm{h}_{L^{2}_{\ol{B}_{1}}} \le \omega\left(\norm{F_{\kappa}(h)}_{\rm HS}\right) \quad \text{for all $h\in H^{1}_{\ol{B}_{1}}$ with $\norm{h}_{H^1} \le 1$,}
\end{equation*}
then 
\begin{equation*}
\omega(t) \gtrsim \max \left\{ t , \kappa^{-1}\left(\log\kappa + \log(1/t)\right)^{-2} \right\} \quad \text{for all $0<t\lesssim 1$,} 
\end{equation*}
where the implied constants are independent of $\kappa$ and $t$.  
\end{theorem}

\addtocontents{toc}{\SkipTocEntry}
\subsection{Methods}

We now explain the methods for proving \Cref{thm:1,thm:2}. Since the Herglotz operator $P_{\kappa}$ satisfies $(P_{\kappa}f)(x)=(f\,\rmd S)\ehat(-\kappa x)$, we have 
\begin{equation}
\norm{A_{\kappa}f}_{L^{2}(B_{1})}^{2} = \frac{1}{\kappa} \int_{B_{\kappa}} \abs{(f\,\rmd S)\ehat(\xi)}^{2} \, \rmd\xi. \label{eq:Agmon-Hormander}
\end{equation}
The right hand side is precisely the expression that appears in Agmon-H\"{o}rmander estimates for Fourier transforms of $L^{2}$-densities. In particular, by \cite[Theorem~3.1]{AgmonHormander}, we have 
\begin{equation}
\frac{1}{R} \int_{B_{R}} \abs{(f\,\rmd S)\ehat(\xi)}^{2} \, \rmd \xi \rightarrow c_{n} \norm{f}_{L^{2}(\mS^{n-1})}^{2} \quad \text{as $R\rightarrow\infty$.} \label{eq:qualitative-inversion}
\end{equation}
The lower bound for singular values of $A_{\kappa}$ in the stable region will be obtained from the Courant max-min principle together with a quantitative version of \eqref{eq:qualitative-inversion} proved in \Cref{lem:quantative-inversion}. The exponential decay rate for singular values in the unstable region is in turn proved by the Courant min-max principle and precise smoothing estimates for $A_{\kappa}$ as an operator from $H^{-s}(\mS^{n-1})$ to $L^{2}(B_{1})$ where $s>0$ is large. 

\begin{remark}\label{rem:Jacobi-Anger}
The singular values $\sigma_{j}(A_{\kappa})$ can be expressed explicitly in terms of Bessel functions. By \eqref{eq:Agmon-Hormander} and the plane wave (or Jacobi-Anger) expansion method, see e.g. \cite[Section~2]{NegroOliveira2023}, one has 
\begin{equation}
\norm{A_{\kappa}f}_{L^{2}(B_{1})}^{2} = (2\pi)^n \sum_{\ell=0}^{\infty} \sum_{m=1}^{N_{\ell}} \Lambda_{\ell}(\kappa) \abs{(f,Y_{\ell,m})}^{2} \label{eq:NegroOliveira2023}
\end{equation}
where $(Y_{\ell,m})$ is an orthonormal basis of $L^{2}(\mS^{n-1})$ consisting of spherical harmonics, $N_{\ell}$ is the dimension of spherical harmonics of degree $\ell$, and the coefficients $\Lambda_{\ell}(\kappa)$ are given by 
\begin{equation*}
\Lambda_{\ell}(\kappa) = \frac{1}{\kappa} \int_{0}^{\kappa} r J_{\ell+\nu}(r)^{2} \, \rmd r,
\end{equation*}
where $J_{\alpha}(z)$ is the Bessel function and $\nu = \frac{n-2}{2}$. The same expressions appear in \cite[(3.2)--(3.3)]{GriesmaierSylvester2017} and \cite[(3.3)--(3.4)]{GriesmaierSylvester2017_3D} for $n=2$ and $n=3$, respectively. The formula \eqref{eq:NegroOliveira2023} implies that $A_{\kappa}^{*}A_{\kappa}$ becomes diagonal in the $(Y_{\ell,m})$ basis, which is therefore a singular value basis for $A_{\kappa}$. The singular values $\sigma_{j}$ are given by the numbers $(c_{n}\Lambda_{\ell}(\kappa))^{1/2}$, when counted with correct multiplicity and arranged in nonincreasing order. As mentioned before, direct estimates for $\Lambda_{\ell}(\kappa)$ are given in \cite{GriesmaierSylvester2017, GriesmaierSylvester2017_3D}, which yield similar singular value estimates as those in \Cref{thm:1} at least when $n=2,3$.
\end{remark}

The argument for $F_{\kappa}$ follows a similar structure. First we observe that the $L^{2}$ norm of the Hilbert-Schmidt operator $F_{\kappa}(h)$ is just the $L^{2}$ norm of its Schwartz kernel. By computing the Schwartz kernel explicitly, we have 
\begin{equation*}
\norm{F_{\kappa}(h)}_{\rm HS}^{2} = \int_{\mS^{n-1}} \int_{\mS^{n-1}} \abs{\hat{h}(\kappa(\omega-\theta))}^{2} \, \rmd S(\omega) \, \rmd S(\theta). 
\end{equation*}
We then give an argument involving the coarea formula to express the double integral as an integral over $B_{2}$. The case $n=3$ is particularly simple, and one obtains 
\begin{equation}
\norm{F_{\kappa}(h)}_{\rm HS}^{2} = c_{n} \int_{B_{2\kappa}} \abs{\hat{h}(\xi)}^{2} \abs{\xi}^{-1} \, \rmd \xi. \label{eq:qualitative-inversion-farfield}
\end{equation}
This expression is somewhat similar to the Agmon-H\"{o}rmander type expression \eqref{eq:qualitative-inversion}, and it also explains why the $H^{-1/2}$ norm is natural in this setting. We can then apply the Courant max-min principle and a quantitative analysis of \eqref{eq:qualitative-inversion-farfield} to estimate the singular values in the stable region. Again the exponential decay estimates in the unstable region follow from an analysis of the smoothing properties of $F_{\kappa}$. We remark that the Hilbert-Schmidt norm was also used in \cite{garde2024linearised} for proving stability in a related problem.

\addtocontents{toc}{\SkipTocEntry}
\subsection{Organization} 

We first prove \Cref{thm:1} in \Cref{sec:singular-Herglotz}, and then prove \Cref{thm:2} in \Cref{sec:Singular-scattering}. In \Cref{sec:instability} we prove \Cref{thm:1-instability,thm:2-instability}. Finally, we give some numerical evidence for the singular value estimates in \Cref{thm:1} and \Cref{thm:2} in \Cref{sec:numerical}.

\addtocontents{toc}{\SkipTocEntry}
\subsection*{Acknowledgments}

\noindent 
The second author would like to thank Plamen Stefanov for suggesting the term \emph{increasing resolution}. The first author was partly supported by the NCCU Office of research and development and the National Science and Technology Council of Taiwan (NSTC~112-2115-M-004-004-MY3). The second author was partly supported by the Academy of Finland (Centre of Excellence in Inverse Modelling and Imaging and FAME Flagship, 312121 and 359208) and by the European Research Council under Horizon 2020 (ERC~CoG~770924). The third author was supported by Key-Area Research and Development Program of Guangdong Province (No.~2021B0101190003), NSFC (No.~11925104) and the Chinese Scholarship Council (No.~202106100094). 

\section{\label{sec:singular-Herglotz}Singular values of Herglotz operator}

Let $0=\lambda_{1} < \lambda_{2} \le \lambda_{3} \le \cdots$ be the eigenvalues of the Laplace-Beltrami operator $-\Delta_{\mS^{n-1}}$ on $\mS^{n-1}$, and let $(\phi_{\ell})_{\ell=1}^{\infty}$ be an orthonormal basis of $L^{2}(\mS^{n-1})$ consisting of eigenfunctions with $-\Delta_{\mS^{n-1}}\phi_{\ell} = \lambda_{\ell}\phi_{\ell}$. Recall that $H^{s}(\mS^{n-1})$, for each $s\in\mR$, is the Sobolev space usually equipped with the norm 
\begin{equation*}
\norm{f}_{H^{s}(\mS^{n-1})}' := \norm{(1-\Delta_{\mS^{n-1}})^{s/2}f}_{L^{2}(\mS^{n-1})} = \left( \sum_{\ell=1}^{\infty} (1+\lambda_{\ell})^{s} \abs{(f,\phi_{\ell})}^{2} \right)^{1/2}.
\end{equation*}
By using Weyl asymptotics (see e.g.\ \cite[Theorem~8.3.1]{Tay11PDEvol2}), there exists a constant $C_{n}>0$ such that 
\begin{equation}
C_{n}^{-1}\ell^{\frac{2}{n-1}} \le 1 + \lambda_{\ell} \le C_{n} \ell^{\frac{2}{n-1}} \quad \text{for all $\ell \geq 1$.} \label{eq:Weyl}
\end{equation}
Thus we have an equivalent norm 
\begin{equation*}
\norm{f}_{H^{s}(\mS^{n-1})} := \left( \sum_{\ell=1}^{\infty} \ell^{\frac{2s}{n-1}} \abs{(f,\phi_{\ell})}^{2} \right)^{1/2},
\end{equation*}
and we will use this norm on $H^{s}(\mS^{n-1})$ hereafter. We will now prove \Cref{thm:1} by showing the estimates \eqref{eq:stable-region-Herglotz} and \eqref{eq:unstable-region-Herglotz} separately. 

\addtocontents{toc}{\SkipTocEntry}
\subsection{The stable region}

In order to study the singular values in the stable region, we prove a quantative version of the estimate in \cite[Theorem~3.1]{AgmonHormander} as follows. 

\begin{lemma}\label{lem:quantative-inversion}
Let $\phi \in C_{c}(\mR^{n})$ be a radially symmetric function, i.e.\ there exists $\varphi \in C_{c}(\mR)$ such that $\phi(x) = \varphi(\abs{x})$. Then there exist a constant $C = C(n,\phi)>0$ such that 
\begin{equation*}
\begin{aligned}
& \left| \frac{1}{R} \int_{\mR^{n}} \abs{(f\,\rmd S)\ehat(\xi)}^{2} \phi(\xi/R) \, \rmd \xi - 2(2\pi)^{n-1} \left( \int_{0}^{\infty} \varphi(r) \, \rmd r \right) \norm{f}_{L^{2}(\mS^{n-1})}^{2} \right| \\
& \quad \le \frac{C}{R} \norm{f}_{L^{2}(\mS^{n-1})} \norm{f}_{H^{1}(\mS^{n-1})}
\end{aligned}
\end{equation*}
whenever $f \in H^{1}(\mS^{n-1})$ and $R \ge 1$. 
\end{lemma}

Before proving \Cref{lem:quantative-inversion}, we show how it can be used to prove \eqref{eq:stable-region-Herglotz}. 

\begin{proof}[Proof of \eqref{eq:stable-region-Herglotz} in \Cref{thm:1}]
Since $\sigma_{1}(A_{\kappa}) = \norm{A_{\kappa}}_{L^{2}(\mS^{n-1})\rightarrow L^{2}(B_{1})}$, the Agmon-H\"ormander estimate \eqref{eq:normalized-AG-inequality} implies that 
\begin{equation}
\sigma_{j} \le C(n) \quad \text{for all $j\ge 1$.} \label{eq:upper-bound-stable-region}
\end{equation}
We now consider the lower bound. Given a fixed small constant $\delta$, we define 
\begin{equation*}
X_{\delta,\kappa} = {\rm span}\,\{ \phi_{\ell} : \ell^{\frac{2}{n-1}} \le \delta \kappa^{2} \}.
\end{equation*}
Then $N := \dim\,(X_{\delta,\kappa}) = \lfloor \delta^{\frac{n-1}{2}} \kappa^{n-1} \rfloor$. For $f\in X_{\delta,\kappa}$, we can write $f = \sum_{\ell=1}^{N}(f,\phi_{\ell})\phi_{\ell}$ where $(f,g)$ denotes the inner product in $L^{2}(\mS^{n-1})$. We see that  
\begin{equation}
\norm{f}_{H^{1}(\mS^{n-1})}^{2} = \sum_{\ell=1}^{N} \ell^{\frac{2}{n-1}} \abs{(f,\phi_{\ell})}^{2} \le \delta \kappa^{2} \sum_{\ell=1}^{N} \abs{(f,\phi_{\ell})}^{2} = \delta\kappa^{2}\norm{f}_{L^{2}(\mS^{n-1})}^{2}. \label{eq:expansion-space-X}
\end{equation}
By using \Cref{lem:quantative-inversion} with a fixed radial function $\psi \in C_{c}(\mR)$ such that $0 \le \psi \le 1$, $\psi(x)=1$ for $\abs{x} \le 1/2$, and $\psi(x)=0$ for $\abs{x} \ge 1$, from \eqref{eq:Agmon-Hormander} we have 
\begin{equation}
\begin{aligned}
& \norm{A_{\kappa}f}_{L^{2}(B_{1})}^{2} \ge \frac{1}{\kappa} \int_{\mR^{n}} \abs{(f\,\rmd S)\ehat(\xi)}^{2} \psi(\abs{\xi}/\kappa) \, \rmd \xi \\
& \quad \ge 2(2\pi)^{n-1} \norm{\psi}_{L^{1}(\mR_{+})} \norm{f}_{L^{2}(\mS^{n-1})}^{2} - \frac{C_{n,\psi}}{\kappa} \norm{f}_{L^{2}(\mS^{n-1})} \norm{f}_{H^{1}(\mS^{n-1})}. \label{eq:application-quantative-inverse} 
\end{aligned}
\end{equation}
We now use \eqref{eq:expansion-space-X} and choose $\delta = \delta_{n,\psi}>0$ sufficiently small such that the right hand side of \eqref{eq:application-quantative-inverse} is greater than 
\begin{equation*}
(2\pi)^{n-1} \norm{\psi}_{L^{1}(\mR_{+})} \norm{f}_{L^{2}(\mS^{n-1})}^{2}. 
\end{equation*}
This gives that 
\begin{equation}
\norm{A_{\kappa}f}_{L^{2}(B_{1})}^{2} \ge c_{n,\psi} \norm{f}_{L^{2}(\mS^{n-1})}^{2} \quad \text{for all $f\in X_{\delta,\kappa}$ and $\kappa \ge 1$} \label{eq:lower-bound-Herglotz-stable-region}
\end{equation}
for some constant $c_{n,\psi} > 0$. For each $j \ge 1$, the Courant max-min principle implies that 
\begin{equation*}
\sigma_{j}(A_{\kappa}) = \max_{X} \min_{\{f\in X \,:\, \norm{f}_{L^{2}(\mS^{n-1})}=1\}} \norm{A_{\kappa}f}_{L^{2}(B_{1})}
\end{equation*}
where the maximum is over all subspaces $X \subset L^{2}(\mS^{n-1})$ with $\dim\,(X) = j$. When $j = \lfloor \delta^{\frac{n-1}{2}} \kappa^{n-1} \rfloor$, we have $\dim(X_{\delta,\kappa}) = j$ and therefore by \eqref{eq:lower-bound-Herglotz-stable-region} 
\begin{equation*}
\sigma_{j}(A_{\kappa}) \ge \min_{\{f\in X_{\delta,\kappa} \,:\, \norm{f}_{L^{2}(\mS^{n-1})=1}\}} \norm{A_{\kappa}f}_{L^{2}(B_{1})} \ge c_{n,\psi}^{1/2}.
\end{equation*}
Combining this with \eqref{eq:upper-bound-stable-region}, we conclude \eqref{eq:stable-region-Herglotz}. 
\end{proof}

It remains to prove the quantitative Agmon-H\"{o}rmander estimate in \Cref{lem:quantative-inversion}. 

\begin{proof}[Proof of \Cref{lem:quantative-inversion}]
Write $u = f \, \rmd S$ and $\Phi_{R}(x) := R^{n-1} \check{\phi}(Rx)$ where $\check{\phi}$ is the inverse Fourier transform of $\phi$. Since $\hat{\Phi}_{R}(\xi) = R^{-1} \phi(\xi/R)$, Parseval's identity yields 
\begin{equation*}
\begin{aligned}
& I:= \frac{1}{R} \int_{\mR^{n}} \abs{\hat{u}(\xi)}^{2} \phi(\xi/R) \, \rmd S = \int_{\mR^{n}} \abs{\hat{u}(\xi)}^{2} \hat{\Phi}_{R}(\xi) \, \rmd \xi \\
& \quad = (2\pi)^{n} \int_{\mR^{n}} (u*\Phi_{R})(x) \overline{u(x)} \, \rmd x = (2\pi)^{n} \int_{\mS^{n-1}} (u*\Phi_{R})(\hat{x}) \overline{f(\hat{x})} \, \rmd S(\hat{x}). 
\end{aligned}
\end{equation*}
By choosing $\chi \in C_{c}^{\infty}(\mR^{n})$ satisfying $0 \le \chi \le 1$, $\chi=1$ in $\overline{B}_{1/4}$ and $\chi=0$ outside $B_{1/2}$, and writing $\Phi_R = \chi \Phi_R + (1-\chi) \Phi_R$, we obtain 
\begin{equation*}
(u*\Phi_{R})(\hat{x}) = I_{1}(\hat{x}) + I_{2}(\hat{x}) + I_{3}(\hat{x}) \quad \text{for all $\hat{x} \in \mS^{n-1}$},
\end{equation*}
where we define 
\begin{equation*}
\begin{aligned}
& I_{1}(\hat{x}) := \int_{\mS^{n-1}} \Phi_{R}(\hat{x}-\hat{y}) (1 - \chi(\hat{x}-\hat{y})) f(\hat{y}) \, \rmd S(\hat{y}), \\
& I_{2}(\hat{x}) := f(\hat{x}) \int_{\mS^{n-1}} (\chi \Phi_{R})(\hat{x}-\hat{y}) \, \rmd S(\hat{y}) \\
& I_{3}(\hat{x}) := \int_{\mS^{n-1}} (\chi \Phi_{R})(\hat{x}-\hat{y}) (f(\hat{y})-f(\hat{x})) \, \rmd S(\hat{y}),
\end{aligned}
\end{equation*}
then we have 

\medskip

\noindent\textbf{We first estimate $I_{1}(\hat{x})$ for each $\hat{x} \in \mS^{n-1}$.} If $\hat{y} \in \supp\,(\tilde{f})$, where $\tilde{f}(\hat{y}) := (1 - \chi(\hat{x}-\hat{y}))f(\hat{y})$, then $\abs{\hat{x}-\hat{y}} \ge 1/4$. Writing $\tilde{\Phi}(z) = \abs{z}^{n} \check{\phi}(z)$, we have 
\begin{equation*}
\begin{aligned}
& \abs{\Phi_{R}(\hat{x}-\hat{y})} \le 4^{n} \abs{\hat{x}-\hat{y}}^{n} \abs{\Phi_{R}(\hat{x}-\hat{y})} \\
& \quad \le \frac{4^{n}}{R} \abs{\tilde{\Phi}(R(\hat{x}-\hat{y})} \le \frac{4^{n}}{R} \norm{\tilde{\Phi}}_{L^{\infty}(\mR^{n})} \quad \text{for all $\hat{y} \in \supp\,(\tilde{f})$,}
\end{aligned}
\end{equation*}
since $\tilde{\Phi}$ is a Schwartz function on $\mR^{n}$. Then 
\begin{equation*}
\sup_{\hat{x}\in\mS^{n-1}} \abs{I_{1}(\hat{x})} \le \frac{C_{n,\phi}}{R} \int_{\mS^{n-1}} \abs{\tilde{f}(\hat{y})} \,\rmd S(\hat{y}) \le \frac{C_{n,\phi}}{R} \norm{f}_{L^{2}(\mS^{n-1})},
\end{equation*}
and thus we derive that 
\begin{equation}
\left| \int_{\mS^{n-1}} I_{1}(\hat{x}) \overline{f(\hat{x})} \, \rmd S(\hat{x}) \right| \le \frac{C}{R} \norm{f}_{L^{2}(\mS^{n-1})}^{2}. \label{eq:estimate-I1}
\end{equation}

\medskip

\noindent\textbf{We now estimate $I_{2}(\hat{x})$.} Let $\mathcal{R}_{\hat{x}} \in {\rm SO}(n)$ be a rotation such that $\mathcal{R}_{\hat{x}}\hat{x} = (0,\cdots,0,1)$. Hence we see that 
\begin{equation*}
\{ \hat{y} \in \mS^{n-1} : \abs{\hat{y}-\hat{x}} < 1/2 \} = \{ \hat{y} = \mathcal{R}_{\hat{x}}^{-1}(y',h(y')) : y'\in B_{c}^{n-1} \},
\end{equation*}
for some $0 < c < 1$, where $B_{r}^{n-1}$ is the ball in $\mR^{n-1}$ with radius $r$ and $h(y') \equiv \sqrt{1-\abs{y'}^{2}}$. From the assumption that $\phi$ is radially symmetric we see that $\check{\phi}$ is also radially symmetric, i.e. $\check{\phi}\circ\mathcal{R}_{\hat{x}}^{-1} = \check{\phi}$, therefore 
\begin{equation*}
\begin{aligned}
& I_{2}(\hat{x}) = f(\hat{x}) \int_{B_{c}^{n-1}} (\chi \Phi_{R})\circ \mathcal{R}_{\hat{x}}^{-1} (-y',h(0)-h(y'))\cdot a(y') \, \rmd y' \\
& \quad = f(\hat{x}) \int_{B_{cR}^{n-1}} (R^{1-n}\chi \Phi_{R})\circ\mathcal{R}_{\hat{x}}^{-1}(-z'/R,h(0)-h(z'/R)) \cdot a(z'/R) \, \rmd z' \\
& \quad = f(\hat{x}) \int_{B_{cR}^{n-1}} \check{\phi} (-z',R(h(0)-h(z'/R))) \times \\
& \qquad \times (\chi\circ\mathcal{R}_{\hat{x}}^{-1})(-z'/R,h(0)-h(z'/R)) a(z'/R) \, \rmd z'
\end{aligned}
\end{equation*}
where $a(y')=(1+\abs{\nabla'h(y')}^{2})^{1/2}$. Since $\check{\phi}$ is Schwartz, $\chi\circ\mathcal{R}_{\hat{x}}^{-1} \in C_{c}^{\infty}(\mR^{n-1})$, $\chi\circ\mathcal{R}_{\hat{x}}^{-1}(0)=1$ and $a$ is smooth with $a(0)=1$, using Taylor's theorem we see that 
\begin{equation*}
\begin{aligned}
& I_{2}(\hat{x}) = f(\hat{x}) \left( \int_{\mR^{n-1}} \check{\phi}(-z',-\nabla'h(0)\cdot z')\,\rmd z' + Q_{R}(\hat{x}) \right) \\
& \quad =f(\hat{x}) \int_{\mR^{n-1}} \check{\phi}(-z',0)\,\rmd z' + f(\hat{x})Q_{R}(\hat{x}) \\
& \quad =f(\hat{x}) (2\pi)^{-1} \int_{\mR} \phi(0,z_{n}) \, \rmd z_{n} + f(\hat{x})Q_{R}(\hat{x}) \\
& \quad =f(\hat{x})\pi^{-1} \int_{0}^{\infty} \varphi(z_{n}) \, \rmd z_{n} + f(\hat{x})Q_{R}(\hat{x})
\end{aligned}
\end{equation*}
with $\sup_{\hat{x}\in\mS^{n-1}}\abs{Q_{R}(\hat{x})} \le C_{n,\phi}/R$. This yields the following estimate involving $I_{2}$: 
\begin{equation}
\begin{aligned}
& \left| \int_{\mS^{n-1}} I_{2}(\hat{x}) \overline{f(\hat{x})} \, \rmd S(\hat{x}) - \pi^{-1} \left( \int_{0}^{\infty} \varphi(r) \, \rmd r \right) \norm{f}_{L^{2}(\mS^{n-1})}^{2} \right| \\
& \quad \le \frac{C}{R} \norm{f}_{L^{2}(\mS^{n-1})}^{2}.
\end{aligned} \label{eq:estimate-I2}
\end{equation}

\medskip

\noindent\textbf{Finally, we want to estimate $I_{3}$.} We make the change of variables $y = \exp_{\hat{x}}(w)$, where $\exp_{\hat{x}}$ is the exponential map on $\mS^{n-1}$, to obtain that 
\begin{equation*}
I_{3}(\hat{x}) = \int_{T_{\hat{x}}\mS^{n-1}} (\chi \Phi_{R})(\hat{x}-\exp_{\hat{x}}(w)) (f(\exp_{\hat{x}}(w)) - f(\hat{x})) \, \rmd T(w). 
\end{equation*}
Here $\rmd T$ is the volume form on $T_{\hat{x}}\mS^{n-1} \cong \mR^{n-1}$ corresponding to $\rmd S$. This is a valid change of variables in $\mathrm{supp}(\chi(\hat{x}-\,\cdot\,))$. We have the Taylor expansion 
\begin{equation*}
\exp_{\hat{x}}(w) = \hat{x} + w + O(\abs{w}^{2}).
\end{equation*}
Thus 
\begin{equation*}
(\chi\Phi_{R})(\hat{x} - \exp_{\hat{x}}(w)) = (\chi\Phi_{R})(-w) + R^{-1} g_{R}(x,w),
\end{equation*}
where the Schwartz decay of $\check{\phi}$, the support of $\chi$ and a short computation give that 
\begin{equation*}
\sup_{(\hat{x},R)\in\mS^{n-1}\times\mR_{\ge 1}} \int_{T_{\hat{x}}\mS^{n-1}} \abs{g_{R}(\hat{x},w)} \, \rmd T(w) \le C_{n,\phi}. 
\end{equation*}
Thus we have 
\begin{equation*}
I_{3}(\hat{x}) = \int_{T_{\hat{x}}\mS^{n-1}} \int_{0}^{1} (\chi\Phi_{R})(-w) \, \rmd f(\exp_{\hat{x}}(tw)) (\partial_{t}\exp_{\hat{x}}(tw)) \, \rmd t \, \rmd T(w) + Q_{R}(\hat{x}),
\end{equation*}
with $\sup_{\hat{x}\in\mS^{n-1}} \abs{Q_{R}(\hat{x})} \le C_{n,\phi} R^{-1}( \norm{f}_{L^2(\mS^{n-1})} + |f(\hat{x})|)$ (this requires another change of variables $w = \exp_{\hat{x}}^{-1}(\hat{y})$). Integrating this against $\overline{f(\hat{x})}$ over $\mS^{n-1}$ gives 
\begin{equation*}
\begin{aligned}
& \left| \int_{\mS^{n-1}} I_{3}(\hat{x}) \overline{f(\hat{x})} \, \rmd S(\hat{x}) \right| \\
& \quad \le \int_{\mS^{n-1}} \int_{\mR^{n-1}} \int_{0}^{1} \abs{(\Phi_{R}\chi)(-w)} \abs{\rmd f(\exp_{\hat{x}}(tw))} \abs{w} \abs{f(\hat{x})} \, \rmd t \, \rmd w \, \rmd S(\hat{x}) \\
& \qquad + \frac{C_{n,\phi}}{R} \norm{f}_{L^{2}(\mS^{n-1})}^{2}
\end{aligned}
\end{equation*}
where we have identified the tangent space of $\mS^{n-1}$ near $\hat{x}$ with $\mR^{n-1}$. The map $\hat{x} \mapsto \exp_{\hat{x}}(w)$ is a local diffeomorphism with uniform bounds on its derivatives when $w$ is sufficiently small. Thus, using Fubini and Cauchy-Schwartz, we obtain 
\begin{equation*}
\begin{aligned}
& \left| \int_{\mS^{n-1}} I_{3}(\hat{x}) \overline{f(\hat{x})} \, \rmd S(\hat{x}) \right| \\
& \quad \le C_{n} \norm{\rmd f}_{L^{2}(\mS^{n-1})} \norm{f}_{L^{2}(\mS^{n-1})} \int_{\mR^{n-1}} \abs{w} \abs{(\Phi_{R}\chi)(-w)} \, \rmd w + \frac{C_{n,\phi}}{R} \norm{f}_{L^{2}(\mS^{n-1})}^{2} \\
& \quad \le \frac{C_{n,\phi}}{R} \norm{\rmd f}_{L^{2}(\mS^{n-1})} \norm{f}_{L^{2}(\mS^{n-1})} + \frac{C_{n,\phi}}{R} \norm{f}_{L^{2}(\mS^{n-1})}^{2} \\
& \quad \le \frac{C_{n,\phi}}{R} \norm{f}_{H^{1}(\mS^{n-1})} \norm{f}_{L^{2}(\mS^{n-1})}.
\end{aligned}
\end{equation*}
This proves 
\begin{equation}
\left| \int_{\mS^{n-1}} I_{3}(\hat{x}) \overline{f(\hat{x})} \, \rmd S(\hat{x}) \right| \le \frac{C_{n,\phi}}{R} \norm{f}_{L^{2}(\mS^{n-1})} \norm{f}_{H^{1}(\mS^{n-1})}. \label{eq:estimate-I3}
\end{equation}

\medskip

\noindent\textbf{Conclusion.} By summing \eqref{eq:estimate-I1}, \eqref{eq:estimate-I2} and \eqref{eq:estimate-I3}, we conclude \Cref{lem:quantative-inversion}. 
\end{proof}

\addtocontents{toc}{\SkipTocEntry}
\subsection{The unstable region}

We now complement the bound in the stable region with a decay rate for the singular values in the unstable region. This will be based on smoothing properties of the operator $A_{\kappa}$. First we need the following lemma. 

\begin{lemma}\label{lem:compute-coefficient}
Let $n \ge 2$, $\kappa \ge 1$, $0 \le r \le 1$ and $\theta \in \mS^{n-1}$. There exists a constant $C=C(n)>0$ such that for any $m\ge 0$ one has 
\begin{equation*}
\norm{(-\Delta_{\mS^{n-1}})^{m}(e^{\bfi\kappa r\br{\theta,\cdot}})}_{L^{2}(\mS^{n-1})} \le C(C(m\kappa))^{2m}.
\end{equation*}
\end{lemma}

\begin{proof}
Write $\lambda = \kappa r$. We will first prove that 
\begin{subequations}
\begin{equation}
(-\Delta_{\mS^{n-1}},\omega)^{m} e^{\bfi\lambda \theta\cdot \omega} = e^{\bfi\kappa\theta \cdot \omega} \sum_{j=0}^{2m} c_{m,j} (\theta\cdot\omega)^{j} \quad \text{for all $\omega\in\mS^{n-1}$.} \label{eq:to-do-lemma}
\end{equation}
with 
\begin{equation}
\abs{c_{m,j}} \le C^{2m+1}(2m)! \lambda^{2m}. \label{eq:to-do-lemma-coefficient}
\end{equation}
\end{subequations}
These facts together with the estimates $\norm{\br{\theta,\cdot}^{j}}_{L^{2}(\mS^{n-1})} \le \abs{\mS^{n-1}}^{1/2}$, $m \le 2^{m}$ and $\lambda \le \kappa$ prove the lemma. 

To show \eqref{eq:to-do-lemma}, we observe that for each $j \ge 0$ one has 
\begin{equation}
\begin{aligned}
& (-\Delta_{\mS^{n-1},\omega})(e^{\bfi\lambda\theta\cdot\omega}(\theta\cdot\omega)^{j}) \\
& \quad = (-\Delta_{y}) \left.\left( e^{\bfi\lambda\theta\cdot\frac{y}{\abs{y}}}(\theta\cdot y/\abs{y})^{j} \right) \right|_{y=\omega} \\
& \quad = e^{\bfi\lambda\theta\cdot\omega} \left( -\lambda^{2}(\theta\cdot\omega)^{j+2} + \bfi\lambda(n-1-j)(\theta\cdot\omega)^{j+1} \right.\\
& \qquad \left. + (\lambda^{2} - j(j-1))(\theta\cdot\omega)^{j} + (-\bfi\lambda j + j(j-1))(\theta\cdot\omega)^{j-1} \right)
\end{aligned} \label{eq:induction1}
\end{equation}
for all $\omega \in \mS^{n-1}$. It follows that there exist constants $c_{m,j}\in\mC$ such that \eqref{eq:to-do-lemma} holds. It remains to estimate the coefficients $c_{m,j}$. 

Applying $-\Delta_{\mS^{n-1},\omega}$ to the identity \eqref{eq:to-do-lemma}, we reach 
\begin{equation}
\begin{aligned}
& \sum_{j=0}^{2m} c_{m,j} (-\Delta_{\mS^{n-1},\omega})(e^{\bfi\lambda\theta\cdot\omega}(\theta\cdot\omega)^{j}) \\
& \quad = (-\Delta_{\mS^{n-1},\omega})^{m+1}e^{\bfi\lambda\theta\cdot\omega} \equiv \sum_{j=0}^{2(m+1)} c_{m+1,j} (e^{\bfi\lambda\theta\cdot\omega}(\theta\cdot\omega)^{j}). \label{eq:induction2}
\end{aligned}
\end{equation}
Combining \eqref{eq:induction1} and \eqref{eq:induction2}, if we additionally set $c_{m,j}\equiv 0$ when $j \notin \{0,\cdots,2m\}$ for simplicity, then the coefficients satisfy the recurrence relation 
\begin{equation}
\begin{aligned}
c_{m+1,j} &= -\lambda^{2} c_{m,j-2} + \bfi \lambda (n-j) c_{m,j-1} + (1 + \lambda^{2} - j(j-1)) c_{m,j} \\
& \quad + (-\bfi\lambda(j+1)+j(j+1))c_{m,j+1}
\end{aligned}
\end{equation}
for all $j=0,\cdots,2m$. 

Clearly \eqref{eq:to-do-lemma-coefficient} holds for $j=0$ and $m=0$. Suppose that \eqref{eq:to-do-lemma-coefficient} holds for $m=\ell$. For each $j=2,3,\cdots,2\ell$, using the induction hypothesis, the trivial estimates 
\begin{equation*}
\begin{aligned}
& \abs{n-j} \le n+2\ell ,\quad \abs{\lambda^{2}-j(j-1)}\le \lambda^{2}(1+4\ell^{2} - 2\ell^{2}), \\
& \abs{-\bfi\lambda(j+1)+j(j+1)} \le r\abs{2\ell+1+4\ell^{2}+2\ell} = \lambda(2\ell+1)^{2},
\end{aligned}
\end{equation*}
we see that 
\begin{equation*}
\begin{aligned}
& \abs{c_{\ell+1},j} \le \lambda^{2\ell+2} C^{2\ell+1} (2\ell)! (2+n+2(2\ell+1)^{2}) \\
& \quad \le \lambda^{2\ell+2} C^{2\ell+1} (2\ell)! (2(2+n)(2\ell+1)^{2}) \\
& \quad \le \lambda^{2\ell+2} 2(2+n) C^{2\ell+1} (2\ell)! (2\ell+1) (2\ell+2) \\
& \quad = \lambda^{2(\ell+1)} C^{2(\ell+1)+1} (2(\ell+1))!
\end{aligned}
\end{equation*}
provided $2(2+n) \le C^{2}$. We can also estimate the other terms in a similar manner, and we conclude the lemma by induction. 
\end{proof}

We use the previous lemma to prove the following smoothing estimate for $A_{\kappa}$. 

\begin{lemma}\label{lem:smoothing-Herglotz}
There is $C=C(n)>0$ such that for any integer $m \ge 0$, one has 
\begin{equation*}
\norm{A_{\kappa}f}_{L^{2}(B_{1})} \le C(Cm\kappa)^{2m} \norm{f}_{H^{-2m}(\mS^{n-1})}
\end{equation*}
for any $f\in L^{2}(\mS^{n-1})$. 
\end{lemma}

\begin{proof}
Let $(\phi_{\ell})$ be an orthonormal basis of $L^{2}(\mS^{n-1})$ consisting of eigenfunctions of $-\Delta_{\mS^{n-1}}$ with eigenvalues $0=\lambda_{1}<\lambda_{2}\le\lambda_{3}\le\cdots$. Let $f \in L^{2}(\mS^{n-1})$ and decompose $f = f_{1}+f_{2}$, where $f_{1}=(f,\phi_{1})\phi_{1}$. Since $\phi_{1}$ is a constant depending on $n$, we have 
\begin{equation*}
\norm{A_{\kappa}f_{1}}_{L^{2}(B_{1})} \le \abs{(f,\phi_{1})} \norm{A_{\kappa}\phi_{1}}_{L^{2}(B_{1})} \le C_{n} \norm{f}_{H^{-2m}(\mS^{n-1})}.
\end{equation*}
To estimate $A_{\kappa}f_{2}$ we define an auxiliary function 
\begin{equation*}
\tilde{f}_{2} := \sum_{\ell=2}^{\infty} \lambda_{\ell}^{-m}(f,\phi_{\ell})\phi_{\ell}
\end{equation*}
so that $(-\Delta_{\mS^{n-1}})^{m}\tilde{f}_{2} = f_{2}$. Integration by parts gives 
\begin{equation*}
(A_{\kappa}f_{2})(x) = \left( A_{\kappa} \left( (\Delta_{\mS^{n-1}})^{m} \tilde{f}_{2} \right) \right)(x) = \int_{\mS^{n-1}} (-\Delta_{\mS^{n-1}})^{m} (e^{\bfi\kappa x\cdot\omega}) \tilde{f}_{2}(\omega) \, \rmd S(\omega). 
\end{equation*}
Therefore using \Cref{lem:compute-coefficient} and the Weyl asymptotics in \eqref{eq:Weyl}, we have 
\begin{equation*}
\norm{A_{\kappa}f_{2}}_{L^{2}(B_{1})} \le C(Cm\kappa)^{2m} \norm{\tilde{f}}_{L^{2}(\mS^{n-1})} \le C(Cm\kappa)^{2m} \norm{f_{2}}_{H^{-2m}(\mS^{n-1})}. 
\end{equation*}
Combining the estimates for $A_{\kappa}f_{1}$ and $A_{\kappa}f_{2}$ proves the lemma. 
\end{proof}

We can now prove the second part of \Cref{thm:1}. 

\begin{proof}[Proof of \eqref{eq:unstable-region-Herglotz} in \Cref{thm:1}]
We use Courant's min-max principle 
\begin{equation*}
\sigma_{j}(A_{\kappa}) = \min_{X} \max_{\{f\perp X \,:\, \norm{f}_{L^{2}(\mS^{n-1})}=1\}} \norm{A_{\kappa}f}_{L^{2}(B_{1})}
\end{equation*}
where the minimum is over all subspaces $X$ of $L^{2}(\mS^{n-1})$ with $\dim\,(X)=j-1$. We let $X_{j-1} = {\rm span}\,\{\phi_{1},\cdots,\phi_{j-1}\}$ and observe that \Cref{lem:smoothing-Herglotz} implies 
\begin{equation*}
\norm{A_{\kappa}f}_{L^{2}(B_{1})} \le C(Cm\kappa)^{2m} j^{-\frac{2m}{n-1}} \norm{f}_{L^{2}(\mS^{n-1})} \quad \text{for all $f\perp X_{j-1}$.}
\end{equation*}
Thus choosing $X=X_{j}$ in the min-max principle yields 
\begin{equation*}
\sigma_{j}(A_{\kappa}) \le C\left(Cm\kappa j^{-\frac{1}{n-1}}\right)^{2m}.
\end{equation*}
Here $m\ge 0$ can be chosen freely. The function $h(t)=C\left(Ct\kappa j^{-\frac{1}{n-1}}\right)^{2t}$ over $t\ge 0$ has a global minimum at $t=t_{0}$ where $Ct_{0}\kappa j^{-\frac{1}{n-1}}=1/e$. If $j \ge (2Ce\kappa)^{n-1}$, then $t_{0}\ge 2$, and choosing $m = \lfloor t_{0} \rfloor \ge t_{0}/2$ yields 
\begin{equation*}
\sigma_{j}(A_{\kappa}) \le h(m) \le h(t_0/2) = C (2e)^{-t_0} \le C \exp \left( -c\kappa^{-1}j^{\frac{1}{n-1}} \right) ,
\end{equation*}
where $C>0$ and $c>0$ only depend on $n$. This proves \eqref{eq:unstable-region-Herglotz} in \Cref{thm:1}. 
\end{proof}

\section{\label{sec:Singular-scattering}Singular values of linearized scattering matrix}

For $\kappa>0$ , recall the linearized scattering matrix $F_{\kappa}$ given by \eqref{fkappa_bilinear}:
\begin{equation*}
(F_{\kappa}(h)f,g)_{L^{2}(\mS^{n-1})} = \kappa^{\frac{n-1}{2}} (h P_{\kappa}f,P_{\kappa}g)_{L^{2}(B_{1})}
\end{equation*}
where $f,g \in L^{2}(\mS^{n-1})$. Since $P_{\kappa} f, P_{\kappa} g \in C^{\infty}(\mR^n)$, from the right hand side we see that $F_{\kappa}(h)$ is well-defined as long as $h$ is a compactly supported distribution.\ Furthermore by the definition of $P_{\kappa}$ we see that
\begin{equation*}
(F_{\kappa}(h)f)(\theta) = \int_{\mS^{n-1}} K_{\kappa}[h](\theta,\omega) f(\omega) \, \rmd S(\omega)
\end{equation*}
where the Schwartz kernel of $F_{\kappa}(h)$ is given by 
\begin{equation} \label{kkappa_kernel}
K_{\kappa}[h](\theta,\omega) = \kappa^{\frac{n-1}{2}} \hat{h}(\kappa(\omega-\theta)).
\end{equation}
The Hilbert-Schmidt norm of $F_{\kappa}(h)$ is equal to the $L^{2}$-norm of $K_{\kappa}[h]$ over $\mS^{n-1}\times\mS^{n-1}$. By using the coarea formula, this norm can be expressed as follows.

\begin{lemma} \label{lem_sk}
For any compactly supported distribution $h$, one has 
\[
\norm{F_{\kappa}(h)}_{\rm HS}^2 = c_n \int_{B_{2\kappa}} \abs{\hat{h}(\xi)}^{2} |\xi|^{-1}(4 - \kappa^{-2} |\xi|^{2})^{\frac{n-3}{2}} \, \rmd \xi.
\]
\end{lemma}

We will prove \Cref{lem_sk} in the end of this section. We will now use it to establish  smoothing estimates for $F_{\kappa}$. In particular, these estimates imply that $F_{\kappa}$ gives rise to a compact operator $H^{-s}_{\ol{B}_1} \to {\rm HS}(L^2(\mathbb{S}^{n-1}))$ for any $s$.

\begin{lemma} \label{lem_fk_smoothing}
There is $C = C(n) > 0$ such that for any $\kappa \geq 1$ and any integer $m \geq 1$, one has 
\[
\norm{F_{\kappa}(h)}_{\rm HS} \leq C(Cm\kappa)^{2m} \kappa^{\alpha(n)} \norm{h}_{H^{-2m}}
\]
whenever $h \in H^{-2m}_{\ol{B}_1}$. Here $\alpha(2) = 1/2$  and $\alpha(n) = 0$ for $n \geq 3$.
\end{lemma}
\begin{proof}
In the following, $C$ denotes a positive constant only depending on $n$ that may change from line to line. We first let $n \geq 3$. From \Cref{lem_sk} we obtain 
\begin{align*}
\norm{F_{\kappa}(h)}_{\rm HS}^2 &\leq C \int_{B_{2\kappa}} |\hat{h}(\xi)|^2 |\xi|^{-1} \,d\xi = C(I_1+I_2)
\end{align*}
where 
\[
I_1 = \int_{B_{1}} |\hat{h}(\xi)|^2 |\xi|^{-1} \,d\xi, \qquad I_2 = \int_{B_{2\kappa} \setminus B_1} |\hat{h}(\xi)|^2 |\xi|^{-1} \,d\xi.
\]
\label{fkappa_est_first} 

To estimate $I_1$, we will use the compact support condition for $h$ to bound $|\hat{h}(\xi)|$. As in \cite[proof of Proposition 8.4.2]{HormanderVol1}, we choose a special cutoff function $\chi = \chi_m \in C^{\infty}_c(B_2)$ with $\chi=1$ near $\ol{B}_1$ and 
\begin{equation} \label{special_cutoff_estimate}
\abs{\p^{\alpha} \chi_m} \leq C(Cm)^{\abs{\alpha}}, \qquad \abs{\alpha} \leq 2m.
\end{equation}
Since $h = \chi h$, we have (up to powers of $2\pi$ that will be included in the constant $C(n)$ later) 
\[
\hat{h}(\xi) = \int \hat{h}(\eta) \hat{\chi}(\xi-\eta) \,d\eta = (\br{\,\cdot\,}^{-2m} \hat{h}, \br{\,\cdot\,}^{2m} \ol{\hat{\chi}(\xi-\,\cdot\,)})_{L^2}.
\]
It follows that 
\begin{equation} \label{hhat_est_first}
\abs{\hat{h}(\xi)} \leq \norm{h}_{H^{-2m}} \norm{\br{\,\cdot\,}^{2m} \hat{\chi}(\xi+\,\cdot\,)}_{L^2}.
\end{equation}
Next we note that 
\begin{align*}
 &\norm{\br{\,\cdot\,}^{2m} \hat{\chi}(\xi+\,\cdot\,)}_{L^2} = \norm{\br{D}^{2m}(e^{-i \langle \,\cdot\,,\xi \rangle} \chi)}_{L^2} \\
 &= \norm{(1+(D - \xi)^2)^m \chi}_{L^2}.
\end{align*}
We expand the last expression as 
\begin{align*}
(1+(D - \xi)^2)^m \chi = \sum_{j=0}^m \binom{m}{j} ((D-\xi)^{2})^{j} \chi.
\end{align*}
Now \eqref{special_cutoff_estimate} gives for $0 \leq \ell \leq m$ that 
\begin{align*}
 &\abs{((D-\xi)^2)^{\ell} \chi} = \abs{\sum_{j_1, \ldots, j_{\ell}=1}^{n} (D_{j_1}-\xi_{j_1})^2 \cdots (D_{j_{\ell}}-\xi_{j_{\ell}})^2 \chi} \\
 &\leq \sum_{j_1, \ldots, j_{\ell}=1}^{n} (|D_{j_1}^2 \cdots D_{j_{\ell}}^2 \chi| + \binom{2\ell}{1} |D_{j_1}^2 \cdots D_{j_{\ell-1}}^2 D_{j_{\ell}}\chi| \cdot \abs{\xi} + \ldots + |\chi| \cdot \abs{\xi}^{2m}) \\
&\leq n^{\ell} (C(Cm)^{2\ell} + \binom{2\ell}{1} C(Cm)^{2\ell-1} \abs{\xi} + \ldots + C(Cm)^0 \abs{\xi}^{2m}) \mathbf{1}_{B_2} \\
 &\leq n^{\ell} C (Cm+\abs{\xi})^{2\ell} \leq C(Cm + C\abs{\xi})^{2\ell} \mathbf{1}_{B_2}
\end{align*}
where $\mathbf{1}_{B_2}$ is the indicator function of $B_2$. Combining the last two expressions gives 
\[
\norm{(1+(D - \xi)^2)^m \chi}_{L^2} \leq \sum_{j=0}^m \binom{m}{j} C(Cm + C\abs{\xi})^{2j} \leq C(Cm + C\abs{\xi})^{2m}.
\]
From \eqref{hhat_est_first} we see that 
\begin{equation} \label{hhatest}
\abs{\hat{h}(\xi)} \leq C(Cm + C\abs{\xi})^{2m} \norm{h}_{H^{-2m}}.
\end{equation}
Then $I_1$ can be estimated for $m \geq 1$ by 
\[
I_1 \leq C \sup_{|\xi| \leq 1} \,|\hat{h}(\xi)|^2 \leq C (Cm)^{4m} \norm{h}_{H^{-2m}}^2.
\]

We proceed to estimate $I_2$ by
\[
I_2 \leq (\sup_{1 \leq \abs{\xi} \leq 2 \kappa} \br{\xi}^{4m} \abs{\xi}^{-1}) \norm{h}_{H^{-2m}}^2 \leq C (C\kappa)^{4m} \norm{h}_{H^{-2m}}^2
\]
Combining the estimates for $I_1$ and $I_2$ yields for $m, \kappa \geq 1$ that 
\[
\norm{F_{\kappa}(h)}_{\rm HS} \leq C(Cm\kappa)^{2m} \norm{h}_{H^{-2m}}.
\]
This is the required estimate for $n \geq 3$.

For $n=2$ we observe that 
\begin{align*} 
\norm{F_{\kappa}(h)}_{\rm HS}^2 &\leq C \left( \sup_{B_{2\kappa}} \,\abs{\hat{h}(\xi)}^2 \right) \int_{B_{2\kappa}} |\xi|^{-1} (4-\kappa^{-2} |\xi|^2)^{-\frac{1}{2}}\, \rmd \xi \\
 &\leq C \left( \sup_{B_{2\kappa}} \,\abs{\hat{h}(\xi)}^2 \right) \kappa 
\end{align*}
upon changing variables $\xi = \kappa \eta$ in the integral. Then using \eqref{hhatest} for $|\xi| \leq 2 \kappa$ gives that for $m, \kappa \geq 1$ 
\[
\norm{F_{\kappa}(h)}_{\rm HS} \leq C(Cm\kappa)^{2m} \kappa^{1/2} \norm{h}_{H^{-2m}}. \qedhere
\]
\end{proof}

We will now combine \Cref{lem_sk,lem_fk_smoothing} to prove the following refined version of \Cref{thm:2}.

\begin{theorem} \label{thm_scattering_refined}
Let $n \ge 2$, $0 \leq s \leq 1/2$ and  $\kappa \ge 1$. The singular values of $F_{\kappa}: H^{-s}_{\ol{B}_1} \to {\rm HS}(L^2(\mathbb{S}^{n-1}))$ satisfy 
\begin{align}
    \sigma_j &\gtrsim j^{-\frac{2s-1}{2n}}, &\qquad j &\lesssim \kappa^{n}, \label{fk_sest1_refined}\\
    \sigma_j &\lesssim \kappa^{\alpha(n)} j^{\frac{s}{n}} \exp(-c j^{\frac{1}{2n}}/\kappa^{1/2}), &\qquad j &\gtrsim \kappa^{n}, \label{fk_sest2_refined}
\end{align}
where the constant $c > 0$ and the implied constants are independent of $\kappa$ and $j$, and $\alpha(n)$ is as in \Cref{lem_fk_smoothing}.
\end{theorem}

For the proof we will need some facts on the singular values of embeddings between Sobolev spaces.

\begin{lemma} \label{lemma_singular_embedding}
If $s_1, s_2 \in \mR$ with $s_1 > s_2$, one has 
\begin{equation} \label{is_est}
\sigma_k(i: H^{s_1}_{\ol{B}_1} \to H^{s_2}_{\ol{B}_1}) \sim k^{-\frac{s_1-s_2}{n}}
\end{equation}
with implicit constants depending on $n$, $s_1$ and $s_2$. Moreover, if $s \in \mR$ and $m \geq 1$ is an integer with $s < 2m$, then 
\begin{equation} \label{is_est2}
\sigma_k(i: H^{-s}_{\ol{B}_1} \to H^{-2m}_{\ol{B}_1}) \leq C_{n,s}(C_n m)^{2m} k^{-\frac{2m-s}{n}}
\end{equation}
for some constants $C_n, C_{n,s} > 0$.
\end{lemma}
\begin{proof}
By \cite[Proposition 3.11]{KRS21InstabilityMechanism} the entropy numbers of $i: H^{s_1}_{\ol{B}_1} \to H^{s_2}_{\ol{B}_1}$ satisfy $e_k(i) \sim k^{-\frac{s_1-s_2}{n}}$. Then \cite[Lemma 3.9]{KRS21InstabilityMechanism} gives \eqref{is_est}. 

Let us next consider $i: H^{-s}_{\ol{B}_1} \to H^{-2m}_{\ol{B}_1}$. Let $\mathbb{T}^n = [-\pi,\pi]^n$ be the torus with opposite sides identified, and let $( \psi_l )_{l=1}^{\infty}$ be an orthonormal basis of $L^2(\mathbb{T}^n)$ consisting of eigenfunctions of $-\Delta$ with eigenvalues $0 = \mu_1 < \mu_2 \leq \ldots$. We write 
\[
i = T \circ j \circ S.
\]
Here $S$ is the periodic extension operator 
\[
S: H^{-s}_{\ol{B}_1} \to H^{-s}(\mathbb{T}^n), \ (Sf, \psi)_{\mathbb{T}^n} = (f, \rho \psi)_{\mR^n}, \qquad \psi \in H^s(\mathbb{T}^n)
\]
where $\rho \in C^{\infty}_c(B_2)$ satisfies $\rho = 1$ near $\ol{B}_1$, $j$ is the inclusion 
\[
j: H^{-s}(\mathbb{T}^n) \to H^{-2m}(\mathbb{T}^n),
\]
and $T$ is given by 
\[
T: H^{-2m}(\mathbb{T}^n) \to H^{-2m}(\mR^n), \ Tf = \chi f
\]
where as in \cite[proof of Proposition 8.4.2]{HormanderVol1}, $\chi = \chi_m \in C^{\infty}_c(\mR^n)$ is a cutoff function with $\chi=1$ near $\ol{B}_1$, $\supp(\chi) \subset B_2$, and 
\begin{equation} \label{special_cutoff_estimate2}
\abs{\p^{\alpha} \chi_m} \leq C(Cm)^{\abs{\alpha}}, \qquad \abs{\alpha} \leq 2m.
\end{equation}

We have 
\[
\sigma_k(i) \leq \norm{T} \sigma_k(j) \norm{S}.
\]
Now as in \cite{KRS21InstabilityMechanism} one has 
\[
(j^* j(u), u)_{H^{-s}(\mathbb{T}^n)} = \norm{u}_{H^{-2m}(\mathbb{T}^n)}^2 = \sum (1+\mu_l)^{-2m} |(u,\psi_l)_{\mathbb{T}^n}|^2 = (Du,u)_{H^{-s}(\mathbb{T}^n)}
\]
where $D$ is the diagonal operator $Du = \sum (1+\mu_l)^{-2m+s} (u, \psi_l)_{\mathbb{T}^n} \psi_l$. Thus $j^* j = D$, and it follows from Weyl asymptotics that 
\[
\sigma_k(j) = (1+\mu_k)^{-m+\frac{s}{2}} \leq C_n^{2m-s} k^{-\frac{2m-s}{n}}.
\]
We also have 
\[
\norm{Sf}_{H^{-s}(\mathbb{T}^n)} \leq \sup_{\norm{\psi}_{H^s(\mathbb{T}^n)}=1} \norm{f}_{H^{-s}} \norm{\rho \psi}_{H^s} \leq C_{n,s} \norm{f}_{H^{-s}}.
\]
Finally, we may use \eqref{special_cutoff_estimate2} to compute 
\[
\norm{\chi g}_{H^{2m}} \leq C(Cm)^{2m} \norm{g}_{H^{2m}}.
\]
By duality this gives for $f \in C^{\infty}_c(B_1)$ that 
\[
\norm{Tf}_{H^{-2m}(\mR^n)} = \sup \frac{(f,\chi g)_{\mathbb{T}^n}}{\norm{g}_{H^{2m}(\mR^n)}} \leq C(Cm)^{2m} \norm{f}_{H^{-2m}(\mathbb{T}^n)}.
\]
Combining these estimates gives \eqref{is_est2}.
\end{proof}

\begin{proof}[Proof of \Cref{thm_scattering_refined}]
We begin by proving \eqref{fk_sest1_refined}. Let first $0 \leq s < 1/2$. By \Cref{lem_sk}, for any $h \in H^{-s}_{\ol{B}_1}$ we have 
\begin{align*}
\norm{F_{\kappa}(h)}_{\rm HS}^2 &\geq c_n \int_{B_{\kappa}} \abs{\hat{h}(\xi)}^2 \abs{\xi}^{-1} \,d\xi \\
 &= c_n \left( \int_{\mR^n} \abs{\hat{h}(\xi)}^2 \abs{\xi}^{-1} \,d\xi - \int_{\mR^n \setminus B_{\kappa}} \abs{\hat{h}(\xi)}^2 \abs{\xi}^{-1} \,d\xi \right) \\
 &\geq c_n \norm{h}_{H^{-1/2}}^2 - C_n \kappa^{2s-1} \norm{h}_{H^{-s}}^2.
\end{align*}
Then the Courant max-min principle gives
\begin{align*}
\sigma_j(F_{\kappa}: H^{-s}_{\ol{B}_1} \to {\rm HS})^2 &= \max_X \min_{h \in X, \norm{h}_{H^{-s}}=1} \norm{F_{\kappa}(h)}_{\rm HS}^2 \\
&\geq \max_X \min_{h \in X, \norm{h}_{H^{-s}}=1} (c_n \norm{h}_{H^{-1/2}}^2 - C_n \kappa^{2s-1} \norm{h}_{H^{-s}}^2),
\end{align*}
where the maximum is over all subspaces $X$ of $H^{-s}_{\ol{B}_1}$ with $\dim(X) = j$. Let $( \psi_l )_{l=1}^{\infty}$ be an orthonormal basis of $H^{-s}_{\ol{B}_1}$ consisting  of singular vectors of $i_s: H^{-s}_{\ol{B}_1} \to H^{-1/2}_{\ol{B}_1}$. By \eqref{is_est}, if $X = \mathrm{span}\{\psi_1, \ldots, \psi_j\}$, then any $h \in X$ satisfies  
\[
\norm{h}_{H^{-1/2}}^2 = \norm{i_s(h)}_{H^{-1/2}}^2 \geq \sigma_j(i_s)^2 \norm{h}_{H^{-s}}^2 \geq \tilde{c}_{n,s} j^{\frac{2s-1}{n}} \norm{h}_{H^{-s}}^2.
\]
Thus when $c_n \tilde{c}_{n,s} j^{\frac{2s-1}{n}} \geq C_n \kappa^{2s-1}/2$, i.e.\ $j \leq c \kappa^n$ with $c$ small enough, we obtain 
\[
\sigma_j(F_{\kappa}: H^{-s}_{\ol{B}_1} \to {\rm HS})^2 \gtrsim j^{\frac{2s-1}{n}}.
\]
This proves \eqref{fk_sest1_refined} for $0 \leq s < 1/2$.

To prove \eqref{fk_sest1_refined} also for $s=1/2$, we use the above argument to conclude that 
\begin{align*}
\sigma_j(F_{\kappa}: H^{-1/2}_{\ol{B}_1} \to {\rm HS})^2 
&\geq \max_X \min_{h \in X, \norm{h}_{H^{-1/2}}=1} (c_n \norm{h}_{H^{-1/2}}^2 - C_n \kappa^{-1} \norm{h}_{L^2}^2),
\end{align*}
where the maximum is over all subspaces $X$ of $H^{-1/2}_{\ol{B}_1}$ with $\dim(X) = j$. We then choose $X = \mathrm{span}\{\psi_1, \ldots, \psi_j\}$ where $( \psi_l )$ is an orthonormal basis of $L^2_{\ol{B}_1}$ consisting of singular vectors of $i: L^2_{\ol{B}_1} \to H^{-1/2}_{\ol{B}_1}$. For $h \in X$, \eqref{is_est} yields 
\begin{align*}
\norm{h}_{L^2}^2 &= \sum_{l=1}^j |(h,\psi_l)_{L^2}|^2, \\
\norm{h}_{H^{-1/2}}^2 &= \norm{i(h)}_{H^{-1/2}}^2 \gtrsim j^{-1/n} \norm{h}_{L^2}^2.
\end{align*}
Then for $j \leq c \kappa^n$ with $c > 0$ sufficiently small we have $\sigma_j(F_{\kappa}: H^{-1/2}_{\ol{B}_1} \to {\rm HS}) \gtrsim 1$, which proves \eqref{fk_sest1_refined} also for $s=1/2$.

We proceed to the proof of \eqref{fk_sest2_refined}. By Courant's min-max principle 
\[
\sigma_j(F_{\kappa}) = \min_{S} \max_{\begin{subarray}{c}h \perp S \\ \|h\|_{H^{-s}_{\ol{B}_1}}=1\end{subarray}} \| F_{\kappa} h \|_{\rm HS} \label{eq:Courant-min-max-fk}
\]
where the minimum is over all subspaces $S$ of $H^{-s}_{\ol{B}_1}$ with $\dim(S) = j-1$. We consider the embedding $i: H^{-s}_{\ol{B}_1} \to H^{-2m}_{\ol{B}_1}$ and let $(\psi_l)$ be an orthonormal basis of $H^{-s}_{\ol{B}_1}$ consisting of singular vectors of $i$. We let $X = \mathrm{span} \{ \psi_1, \ldots, \psi_{j-1} \}$. Then \eqref{is_est2} ensures that 
\[
\norm{h}_{H^{-2m}_{\ol{B}_1}} = \norm{i(h)}_{H^{-2m}_{\ol{B}_1}} \leq C(Cm)^{2m} j^{-\frac{2m-s}{n}} \norm{h}_{H^{-s}_{\ol{B}_1}}, \qquad h \perp X.
\]
By \Cref{lem_fk_smoothing}, we have   
\begin{align*}
\norm{F_{\kappa}h}_{\rm HS} &\leq C(Cm\kappa)^{2m} \kappa^{\alpha(n)} \norm{h}_{H^{-2m}_{\ol{B}_1}} \\
 &\leq C(Cm\kappa^{1/2})^{4m} \kappa^{\alpha(n)} j^{-\frac{2m-s}{n}} \norm{h}_{H^{-s}_{\ol{B}_1}}, \qquad h \perp X.
\end{align*}
Thus choosing $S = X$ in the min-max principle yields 
\[
\sigma_j(F_{\kappa}) \leq C(Cm\kappa^{1/2})^{4m} \kappa^{\alpha(n)} (j^{-\frac{1}{2n}})^{4m} j^{\frac{s}{n}}.
\]
Here $m \geq 1$ can be chosen freely. The function $f(t) = C(Ct\kappa^{1/2} j^{-\frac{1}{2n}})^{4t} \kappa^{\alpha(n)} j^{\frac{s}{n}}$ over $t \geq 0$ has a global minimum at $t=t_0$ where $Ct_0\kappa^{1/2} j^{-\frac{1}{2n}} = 1/e$. If $j \geq (2Ce \kappa^{1/2})^{2n}$, then $t_0 \geq 2$, and choosing $m = \lfloor t_0 \rfloor \geq t_0/2$ yields 
\begin{align*}
\sigma_j(F_{\kappa}) &\leq f(m) \leq f(t_0/2) = C (2e)^{-2t_0} \kappa^{\alpha(n)} j^{\frac{s}{n}} \\
 &\leq C \kappa^{\alpha(n)} j^{\frac{s}{n}} \exp(-c j^{\frac{1}{2n}}/\kappa^{1/2})
\end{align*}
where $C, c > 0$ only depend on $n$. This proves \eqref{fk_sest2_refined}.
\end{proof}

We conclude this section with a proof of \Cref{lem_sk}, which follows from the next result. 

\begin{lemma}\label{lem:coarea-general}
Let $n \ge 2$ be an integer. If the mapping $(\hat{z},\hat{x})\in\mS^{n-1}\times\mS^{n-1} \mapsto h(\hat{z}-\hat{x})$ is in $L^{1}(\mS^{n-1}\times\mS^{n-1})$, then 
\begin{subequations}
\begin{equation}
\begin{aligned}
& \int_{\mS^{n-1}}\int_{\mS^{n-1}} h(\hat{z}-\hat{x}) \sqrt{1-(\hat{z}\cdot\hat{x})^{2}} \, \rmd S(\hat{z}) \, \rmd S(\hat{x}) \\
& \quad = \frac{2^{3-n}\pi^{\frac{n-1}{2}}}{\Gamma(\frac{n-1}{2})} \int_{B_{2}} h(y)(4-\abs{y}^{2})^{\frac{n-2}{2}} \, \rmd y,
\end{aligned} \label{eq:coarea1}
\end{equation}
and 
\begin{equation}
\begin{aligned}
& \int_{\mS^{n-1}}\int_{\mS^{n-1}} h(\hat{z}-\hat{x}) \, \rmd S(\hat{z}) \, \rmd S(\hat{x}) \\
& \quad = \frac{2^{4-n}\pi^{\frac{n-1}{2}}}{\Gamma(\frac{n-1}{2})} \int_{B_{2}} h(y)\abs{y}^{-1}(4-\abs{y}^{2})^{\frac{n-3}{2}} \, \rmd y,
\end{aligned} \label{eq:coarea2}
\end{equation}
\end{subequations}
where $B_{2} = \{ y\in\mR^{n} : \abs{y} < 2 \}$. 
\end{lemma}

\begin{proof}[Proof of \Cref{lem_sk}]
For any compactly supported distribution $h$, \eqref{kkappa_kernel} gives 
\[
\norm{F_{\kappa}(h)}_{\rm HS}^2 = \int_{\mS^{n-1}}\int_{\mS^{n-1}} 
|K_{\kappa}[h](\theta,\omega)|^2 \, \rmd \theta \, \rmd \omega = \kappa^{n-1} \int_{\mS^{n-1}}\int_{\mS^{n-1}} |\hat{h}(\kappa(\omega-\theta))|^2 \, \rmd \theta \, \rmd \omega.
\]
Applying \Cref{lem:coarea-general} and changing variables, this is equal to 
\[
c_n \kappa^{n-1} \int_{B_2} |\hat{h}(\kappa y)|^2 |y|^{-1} (4-|y|^2)^{\frac{n-3}{2}} \,\rmd y = c_n \int_{B_{2\kappa}} \abs{\hat{h}(\xi)}^{2} |\xi|^{-1}(4 - \kappa^{-2} |\xi|^{2})^{\frac{n-3}{2}} \, \rmd \xi. \qedhere
\]
\end{proof}

The key to prove \Cref{lem:coarea-general} is to interpret $\hat{z} - \hat{x}$ as an element of $\ol{B}_2$ and to use the coarea formula on Riemannian manifolds, which can be found in \cite[Exercise~III.12]{Chavel06RiemannianGeometry}. 

\begin{lemma}\label{lem:Chavel-lemma}
Let $\mathcal{M},\mathcal{N}$ be $C^{r}$ Riemannian manifolds such that $m=\dim\,(\mathcal{M}) \ge \dim\,(\mathcal{N})=n$ and $r>m-n$, and let $\Phi:\mathcal{M}\rightarrow\mathcal{N}$ be a  $C^{r}$ function. Then for any measurable function $g:\mathcal{M}\rightarrow\mR$, which is everywhere non-negative or is in $L^{1}(\mathcal{M})$, one has 
\begin{equation*}
\int_{\mathcal{M}} g \mathsf{J}_{\Phi} \, \rmd V_{m} = \int_{\mathcal{N}} \left( \int_{\Phi^{-1}(y)} g|_{\Phi^{-1}(y)} \, \rmd V_{m-n} \right) \, \rmd V_{n}(y),
\end{equation*}
where for any $k$, $\rmd V_{k}$ denotes the $k$-dimensional volume form and 
\begin{equation}
\mathsf{J}_{\Phi}(p) = \left\{\begin{aligned}
& \abs{\det\,(\rmd \Phi|_{(\ker\,\rmd \Phi|_{p})^{\perp}})} && \text{when} \quad {\rm rank}\,(\rmd \Phi|_{p})=n, \\
& 0 && \text{when} \quad {\rm rank}\,(\rmd \Phi|_{p})<n,
\end{aligned}\right. \label{eq:Jacobian-formula}
\end{equation}
where $\rmd\Phi|_{p} : T_{p}\mathcal{M}\rightarrow T_{\Phi(p)}\mathcal{N}$ is the tangent map of $\Phi:\mathcal{M}\rightarrow\mathcal{N}$ at $p$. 
\end{lemma}

Besides, the following elementary identity is also needed to derive \Cref{lem:coarea-general}: 

\begin{lemma}\label{eq:matrix-identity}
For any vectors $u,v\in\mR^{m}$, we have 
\begin{equation*}
\begin{aligned}
& \det\,(2I_{m}-u\otimes u-v\otimes v) \\
& \quad = 2^{m-1}(2 - u\cdot u - v\cdot v + \frac{1}{2}(u\cdot u)(v\cdot v) - \frac{1}{2}(v\cdot u)^{2}),
\end{aligned}
\end{equation*}
where $I_{m}$ is the $m\times m$ identity matrix, $a\otimes b=ab^{\intercal}$ is the juxtaposition of the vectors and $a\cdot b=a^{\intercal}b$ is the inner product of the vectors. 
\end{lemma}

\begin{proof}
Since both sides are continuous in $u$, we only need to prove the result when $u\cdot u\neq 2$. Using the matrix determinant lemma twice, we have 
\begin{equation*}
\begin{aligned}
& \det\,(2I_{m} - u\otimes u - v\otimes v) \\
& \quad = (1 - v^{\intercal}(2I_{m}-u\otimes u)^{-1}v) \det\,(2I_{m} - u\otimes u) \\
& \quad = 2^{m-1}(2- u^{\intercal}u)(1 - v^{\intercal}(2I_{m}-u\otimes u)^{-1}v).
\end{aligned}
\end{equation*}
On the other hand, by using the Sherman-Morrison formula, we see that 
\begin{equation*}
(2I_{m}-u\otimes u)^{-1} = \frac{1}{2}I_{m} + \frac{u\otimes u}{4-2u^{\intercal}u}.
\end{equation*}
Combining the above two equations we reach 
\begin{equation*}
\begin{aligned}
& \det\,(2I_{m}-u\otimes u-v\otimes v) \\
& \quad = 2^{m-1}(2-u^{\intercal}u) \left( 1 - v^{\intercal} \left( \frac{1}{2}I_{m} + \frac{u\otimes u}{4-2u^{\intercal}u}\right)v\right) \\
& \quad = 2^{m-1}(2 - u^{\intercal}u) \left( 1 - \frac{1}{2}v^{\intercal}v - \frac{\frac{1}{2}(v^{\intercal}u)^{2}}{2-u^{\intercal}u} \right) \\
& \quad = 2^{m-1}(2 - u^{\intercal}u - v^{\intercal}v + \frac{1}{2}(u^{\intercal}u)(v^{\intercal}v) - \frac{1}{2}(v^{\intercal}u)^{2})
\end{aligned}
\end{equation*}
which is our desired lemma. 
\end{proof}

\begin{proof}[Proof of \Cref{lem:coarea-general}]
We consider $\mS^{n-1}$ as a submanifold of $\mR^{n}$ by writing 
\begin{equation*}
\begin{aligned}
\mS^{n-1} &= \{\hat{x}\in\mR^{n} : \abs{\hat{x}}=1 \}, \\
T_{\hat{x}}\mS^{n-1} &= \{\hat{x}^{\perp}\in\mR^{n} : \hat{x}^{\perp}\cdot\hat{x}=0 \}.
\end{aligned}
\end{equation*}
Let $\Phi:\mS^{n-1}\times\mS^{n-1}\rightarrow\mR^{n}$ be given by $\Phi(\hat{z},\hat{x})=\hat{z}-\hat{x}$. For each $y\in B_{2}\setminus\{0\} \subset \mR^{n}$ we first investigate the set $\Phi^{-1}(y)$. Write $\hat{y} = y/\abs{y}$. If $(\hat{z},\hat{x})\in\Phi^{-1}(y)$, the relations $|y|^2 = |\hat{z}-\hat{x}|^2 = 2 - 2 \hat{z} \cdot \hat{x}$ and $\hat{x} \cdot y = \hat{x} \cdot \hat{z} - 1$ give that 
\begin{equation}
\hat{x}\cdot\hat{y} = -\frac{\abs{y}}{2} ,\qquad \hat{z} = \hat{x} + y. \label{eq:coordinate}
\end{equation}
The first equation in \eqref{eq:coordinate} states that the projection of $\Phi^{-1}(y)$ in the second variable $\hat{x}$ 
is isometric to an $(n-2)$-sphere of radius $\sqrt{1-\frac{\abs{y}^{2}}{4}}$, given by 
\[
S(n,\abs{y}) = \{ \hat{x} \in \mS^{n-1} \,:\, \hat{x}\cdot\hat{y} = -\frac{\abs{y}}{2} \}.
\]
Note that when $n=2$, $S(2,\abs{y})$ consists of two points at distance $\sqrt{4-\abs{y}^{2}}$. And once $\hat{x}$ satisfying the first equation in \eqref{eq:coordinate} is fixed, $\hat{z}$ is uniquely determined through the second equation in \eqref{eq:coordinate}.

The above discussion shows that $\Phi^{-1}(y) = F(S(n,|y|))$ where $F$ is the diffeomorphism 
\[
F: S(n,|y|) \to \Phi^{-1}(y), \ F(\hat{x}) = (\hat{x}, \hat{x}+y).
\]
If $\hat{x} \in S(n,|y|)$ and $\{e_1, \ldots, e_{n-2} \}$ is an orthonormal basis of $T_{\hat{x}} S(n,|y|)$, we have $dF(e_j) = (e_j,e_j)$ where the vectors $(e_j,e_j)$ are orthogonal with length $\sqrt{2}$. Therefore 
\[
J_F(\hat{x}) = |\det(dF|_{\hat{x}})| = 2^{\frac{n-2}{2}}.
\]
We can compute the volume of $\Phi^{-1}(y)$ by changing variables \cite[Lemma 8.1.8]{PSU23book} as 
\begin{align}
 & \int_{\Phi^{-1}(y)} \, \rmd V_{n-2} = \int_{F(S(n,|y|))} \, \rmd V_{n-2} = \int_{S(n,|y|)} J_F \, \rmd V_{n-2} = 2^{\frac{n-2}{2}}V_{n-2}(S(n,\abs{y})) \label{eq:volume-form} \\
 & \qquad = 2^{\frac{n-2}{2}} \left( 1-\frac{|y|^2}{4} \right)^{\frac{n-2}{2}} V_{n-2}(\mS^{n-2}) = \frac{2^{2-\frac{n}{2}} \pi^{\frac{n-1}{2}}}{\Gamma(\frac{n-1}{2})} (4-|y|^2)^{\frac{n-2}{2}} \notag
\end{align}
where $V_{n-2}$ denotes the volume induced by the Euclidean metric.

We now compute $\mathsf{J}_{\Phi}$. One can verify that 
\begin{equation*}
\rmd \Phi_{(\hat{z},\hat{x})} (\hat{z}^{\perp},\hat{x}^{\perp}) = \hat{z}^{\perp} - \hat{x}^{\perp}
\end{equation*}
for all $(\hat{z}^{\perp},\hat{x}^{\perp}) \in T_{\hat{z}}\mS^{n-1}\times T_{\hat{x}}\mS^{n-1} \cong T_{(\hat{z},\hat{x})}(\mS^{n-1}\times\mS^{n-1})$. Note that 
\begin{equation}
{\rm rank}\,(\rmd \Phi|_{(\hat{z},\hat{x})})=n ,\quad \text{if and only if} \quad \hat{z}\neq\pm\hat{x}. \label{eq:full-rank-condition}
\end{equation}
Now we let $\hat{z}\neq\pm\hat{x}$. If $\{ e_1, \ldots, e_{n-2} \}$ is an orthonormal basis of $\{ \hat{x} \}^{\perp} \cap \{ \hat{z} \}^{\perp} \subset \mR^n$, we see that there are unit vectors $\alpha,\beta \in \mR^{n}$ such that one has orthonormal bases 
\begin{equation*}
\begin{aligned}
& T_{\hat{z}}\mS^{n-1} = {\rm span}\,\{e_{1},\cdots,e_{n-2},\alpha\}, \\
& T_{\hat{x}}\mS^{n-1} = {\rm span}\,\{e_{1},\cdots,e_{n-2},\beta\}. 
\end{aligned}
\end{equation*}
With these orthonormal bases at hand, one sees that 
\begin{equation}
\begin{aligned}
& I_{n} = e_{1}\otimes e_{1} + \cdots + e_{n-2}\otimes e_{n-2} + \alpha\otimes\alpha + \hat{z}\otimes\hat{z}, \\
& I_{n} = e_{1}\otimes e_{1} + \cdots + e_{n-2}\otimes e_{n-2} + \beta\otimes\beta + \hat{x}\otimes\hat{x}.
\end{aligned} \label{eq:matrix-observation}
\end{equation}
Moreover, by a dimension count we have 
\begin{equation*}
\begin{aligned}
\ker\,(\rmd\Phi|_{(\hat{z},\hat{x})}) &= {\rm span}\,\{(e_{1},e_{1})/\sqrt{2},\cdots,(e_{n-2},e_{n-2})/\sqrt{2}\}, \\
\ker\,(\rmd\Phi|_{(\hat{z},\hat{x})})^{\perp} &= {\rm span}\, \left\{\begin{aligned}  
& (e_{1},-e_{1})/\sqrt{2},\cdots,(e_{n-2},-e_{n-2})/\sqrt{2}, \\
& (\alpha,0),(0,-\beta)
\end{aligned}\right\}.
\end{aligned}
\end{equation*}
In the last two formulas, the bases on the right hand side are orthonormal. Consequently, we can compute $\rmd\Phi|_{\ker\,(\rmd\Phi|_{(\hat{z},\hat{x})})^{\perp}}$ as follows: 
\begin{equation*}
\begin{aligned}
 & \rmd\Phi|_{(\hat{z},\hat{x})}((e_{j},-e_{j})/\sqrt{2}) = \sqrt{2} e_{j} , \quad \text{for $j=1,\cdots,n-2$,} \\
 & \rmd\Phi|_{(\hat{z},\hat{x})}(\alpha,0) = \alpha,  \\
 & \rmd\Phi|_{(\hat{z},\hat{x})}(0,-\beta) = \beta.
\end{aligned}
\end{equation*}
By using \eqref{eq:Jacobian-formula}, the fact $\abs{\det\,(A)}^2= \det\,(AA^{\intercal})$, \eqref{eq:matrix-observation} and \Cref{eq:matrix-identity}, one has 
\begin{equation*}
\begin{aligned}
& \mathsf{J}_{\Phi}(\hat{z},\hat{x})^{2} = \abs{\det\,(\sqrt{2} e_{1}, \ldots, \sqrt{2} e_{n-2},\alpha,\beta)}^{2}\\
& \quad = \det\,(2(e_{1}\otimes e_{1}) + \ldots + 2(e_{n-2}\otimes e_{n-2}) + \alpha\otimes\alpha + \beta\otimes\beta) \\
& \quad =  \det\,(2I_{n} - \hat{z}\otimes\hat{z} - \hat{x}\otimes\hat{x}) \\
& \quad = 2^{n-2} (1 - (\hat{z}\cdot\hat{x})^{2}).
\end{aligned}
\end{equation*}
By continuity, we obtain that 
\begin{equation}
\mathsf{J}_{\Phi}(\hat{z},\hat{x}) = 2^{\frac{n}{2}-1} \sqrt{1 - (\hat{z}\cdot\hat{x})^{2}} \quad \text{for all $(\hat{z},\hat{x})\in\mS^{n-1}\times\mS^{n-1}$.} \label{eq:Jacobian}
\end{equation}

Plugging \eqref{eq:volume-form} and \eqref{eq:Jacobian} into \Cref{lem:Chavel-lemma}, we reach 
\begin{equation*}
\begin{aligned}
& \int_{\mS^{n-1}}\int_{\mS^{n-1}} h(\hat{z}-\hat{x}) \sqrt{1-(\hat{z}\cdot\hat{x})^{2}} \, \rmd S(\hat{z}) \, \rmd S(\hat{x}) \\
& \quad = 2^{1-\frac{n}{2}} \int_{\mS^{n-1}} \int_{\mS^{n-1}} h(\hat{z}-\hat{x}) \mathsf{J}_{\Phi}(\hat{z},\hat{x}) \, \rmd S(\hat{z}) \, \rmd S(\hat{x}) \\
& \quad = 2^{1-\frac{n}{2}} \int_{B_{2}} h(y) \left( \int_{\Phi^{-1}(y)} \, \rmd V_{n-2} \right) \, \rmd y \\
& \quad = \frac{2^{3-n} \pi^{\frac{n-1}{2}}}{\Gamma(\frac{n-1}{2})}  \int_{B_{2}} h(y)(4-\abs{y}^{2})^{\frac{n-2}{2}} \,\rmd y,
\end{aligned}
\end{equation*}
which concludes the proof of \eqref{eq:coarea1}. 

Finally, we note that $\abs{\Phi(\hat{z},\hat{x})}^{2} = 2 - 2\hat{z}\cdot\hat{x}$ and therefore 
\begin{equation*}
\abs{\Phi(\hat{z},\hat{x})}^{2} (4-\abs{\Phi(\hat{z},\hat{x})}^{2}) = (2-2\hat{z}\cdot\hat{x})(2 + 2\hat{z}\cdot\hat{x}) = 4(1-(\hat{z}\cdot\hat{x})^{2}).
\end{equation*}
Based on this observation, by substituting $h(y)=2 \tilde{h}(y)\abs{y}^{-1}(4-\abs{y}^2)^{-\frac{1}{2}}$ in \eqref{eq:coarea1}, we conclude \eqref{eq:coarea2}.
\end{proof}

\section{\label{sec:instability}From singular values to instability estimates}

The relation of singular values and the stability issue for the abstract linear inverse problems has been widely discussed. See for example \cite{EKN89TikhonovRegularization}, \cite[Appendix~A]{KRS21InstabilityMechanism} and references therein. In short, the decay rate of the singular values for the direct operators is related to the stability or instability for the corresponding inverse problem. As we have seen in our results for the specific problems, although the singular values tend to zero exponentially, the singular values in the stable region remain uniformly bounded from below, indicating  stable recovery in some subspaces. In this section, we will refine the instability estimates in \cite[Appendix~A]{KRS21InstabilityMechanism} to take into account the increasing resolution phenomenon for some inverse problems. 

Let $X,Y$ be two separable Hilbert spaces and let $T:X\rightarrow Y$ be a bounded compact injective linear operator. In this case, there exists a sequence of singular values $(\sigma_{j})$ with $\sigma_{1} \ge \sigma_{2} \ge \cdots \rightarrow 0$. 
If we interpret $T$ as the forward operator, then the corresponding inverse problem for $T$ is given by the following: 
\begin{equation}
\text{Given $g\in Y$, determine $f\in X$ with $Tf=g$.} \label{eq:IP}
\end{equation}

As $T$ is a compact operator, the stability for the inverse problem \eqref{eq:IP} can be very poor. To study the ill-posedness of \eqref{eq:IP} we will consider the case where the unknown $f$ belongs to a compact set that  represents the additional a priori bounds available in the problem. Given any parameter $\gamma_{0}>0$ and any orthonormal basis $(\phi_{k}) \subset X$ (not necessarily related to a singular value basis), we define the set 
\begin{equation}
K_{\gamma_{0},(\phi_{k})} = \left\{ f\in X : \norm{f}_{\gamma_{0},(\phi_{k})} := \left( \sum_{j=1}^{\infty} j^{2\gamma_{0}} \abs{(f,\phi_{j})_{X}}^{2} \right)^{1/2} \le 1 \right\}, \label{set:K-gamma}
\end{equation}
which is compact in $X$, see e.g.\ \cite[Lemma~A.1]{KRS21InstabilityMechanism}. We now refine the abstract instability result in \cite[Lemma~A.7]{KRS21InstabilityMechanism} in the following lemma. 

\begin{lemma}\label{lem:refine-KRS}
Let $X,Y$ be two separable Hilbert spaces and let $T:X\rightarrow Y$ be a compact injective linear operator. Given any orthonormal basis $(\phi_{j}) \subset X$ and any $\gamma_{0}\ge 1$, we consider the set $K_{\gamma_{0},(\phi_{j})}$ defined in \eqref{set:K-gamma}. Suppose that the singular values of $T:X\rightarrow Y$ satisfy 
\begin{equation}
\sigma_{j}(T) \le \min\{h_{1},h_{2}\exp(-\mu j^{\beta})\} \label{eq:singular-value} 
\end{equation}
for some parameters $\beta>0$ and $\mu>0$. If there exists a non-decreasing function $t\in\mR_{+} \mapsto \omega(t)\in\mR_{+}$ such that $T|_{K_{\gamma_{0},(\phi_{j})}}$ is $\omega$-stable in the sense of 
\begin{equation*}
\norm{f}_{X} \le \omega (\norm{Tf}_{Y}) \quad \text{for all $f\in K_{\gamma_{0},(\phi_{j})}$,}
\end{equation*}
then 
\begin{equation*}
\omega(t) \ge \max \left\{ h_{1}^{-1} t , 2^{-\gamma_{0}}\mu^{\gamma_{0}/\beta}(\log h_{2} + \log(1/t))^{-\gamma_{0}/\beta} \right\}
\end{equation*}
for all $0<t<\min\{h_{1}2^{-\gamma_{0}},h_{2}e^{-\mu}\}$. 
\end{lemma}

\begin{proof}
We define $N_{\gamma_{0}}(\epsilon) := \lfloor \epsilon^{-1/\gamma_{0}} \rfloor$ and observe that 
\begin{equation*}
N_{\gamma_{0}}(\epsilon) \ge \frac{1}{2} \epsilon^{-1/\gamma_{0}} \quad \text{for all $0 < \epsilon < 2^{-\gamma_{0}}$.}
\end{equation*}
By choosing $j=N_{\gamma_{0}}(\epsilon)$ in \eqref{eq:singular-value}, we see that 
\begin{equation*}
\begin{aligned}
\epsilon \sigma_{N_{\gamma_{0}}(\epsilon)} &\le \epsilon \min \{ h_{1} , h_{2} \exp(-\mu N_{\gamma_{0}}(\epsilon)^{\beta}) \} \\
& \le \min \{ h_{1}\epsilon , h_{2} \exp(-\mu 2^{-\beta} \epsilon^{-\beta/\gamma_{0}}) \}
\end{aligned}
\end{equation*}
for all $0 < \epsilon < 2^{-\gamma_{0}} \le 1$. By using \cite[(A.5)]{KRS21InstabilityMechanism}, there exists $f_{\epsilon} \in {\rm span}\,\{\phi_{j}\}_{j=1}^{N_{\gamma_{0}}(\epsilon)}$ with $\norm{f_{\epsilon}}_{X}=\epsilon$ such that 
\begin{equation*}
\begin{aligned}
\norm{Tf_{\epsilon}}_{Y} &\le \sigma_{N_{\gamma_{0}}(\epsilon)}\norm{f_{\epsilon}}_{X} = \epsilon \sigma_{N_{\gamma_{0}}(\epsilon)} \\
&  \le \min \{ h_{1}\epsilon , h_{2} \exp(-\mu 2^{-\beta} \epsilon^{-\beta/\gamma_{0}}). \}
\end{aligned}
\end{equation*}
We also have $\norm{f_{\eps}}_{\gamma_{0},(\phi_{k})} \leq N_{\gamma_{0}}(\epsilon)^{\gamma_0} \norm{f_{\epsilon}}_{X} \leq 1$. Since $T|_{K_{\gamma_{0}},(\phi_{j})}$ is $\omega$-stable, we see that 
\begin{equation}
\epsilon = \norm{f_{\epsilon}}_{X} \le \omega (\norm{Tf_{\epsilon}}_{Y}) \le \omega \left( \min \{ h_{1}\epsilon , h_{2} \exp(-\mu 2^{-\beta} \epsilon^{-\beta/\gamma_{0}}) \} \right). \label{eq:abstract-instability1}
\end{equation}
Assuming that $0<t<\min\{h_{1}2^{-\gamma_{0}},h_{2}e^{-\mu}\}$, we can choose 
\begin{equation*}
\epsilon = \max\left\{ h_{1}^{-1}t , 2^{-\gamma_{0}}\mu^{\gamma_{0}/\beta} (\log h_{2} + \log(1/t))^{-\gamma_{0}/\beta} \right\} \in (0,2^{-\gamma_{0}}).
\end{equation*}
\textbf{Case 1.} If $\epsilon = h_{1}^{-1}t$, then \eqref{eq:abstract-instability1} implies that 
\begin{equation*}
h_{1}^{-1}t \le \omega(h_{1}\epsilon) = \omega(t).
\end{equation*}
\textbf{Case 2.} If $\epsilon=2^{-\gamma_{0}}\mu^{\gamma_{0}/\beta} (\log h_{2} + \log(1/t))^{-\gamma_{0}/\beta}$, then \eqref{eq:abstract-instability1} implies that 
\begin{equation*}
2^{-\gamma_{0}}\mu^{\gamma_{0}/\beta} (\log h_{2} + \log(1/t))^{-\gamma_{0}/\beta} \le \omega(h_{2} \exp(-\mu 2^{-\beta} \epsilon^{-\beta/\gamma_{0}})) =\omega(t).
\end{equation*}
Combining the above two cases, we conclude our lemma. 
\end{proof}

We are now ready to prove \Cref{thm:1-instability}. 

\begin{proof}[Proof of \Cref{thm:1-instability}]
We choose $X=L^{2}(\mS^{n-1})$, $Y = L^2(B_1)$, $T = A_{\kappa}$ and $(\phi_{j})$ as the eigenfunctions corresponding to $-\Delta_{\mS^{n-1}}$. These choices lead us to 
\begin{equation*}
\norm{f}_{\gamma_{0},(\phi_{j})} = \norm{f}_{H^1(\mS^{n-1})} \quad \text{with $\gamma_{0}=\frac{1}{n-1}$.}
\end{equation*}
From \Cref{thm:1}, we have 
\begin{equation*}
\sigma_{j}(A_{\kappa}) \lesssim \min\left\{1,\exp\left(-c\kappa^{-1}j^{\frac{1}{n-1}}\right)\right\},
\end{equation*}
which verifies \eqref{eq:singular-value} with 
\begin{equation*}
h_{1} \sim 1 ,\quad h_{2} \sim 1 ,\quad \mu = c\kappa^{-1} ,\quad \beta = \frac{1}{n-1}. 
\end{equation*}
By using \Cref{lem:refine-KRS}, we then see that 
\begin{equation*}
\omega(t) \gtrsim \max \left\{ t , c\kappa^{-1}(C + \log(1/t))^{-1} \right\}
\end{equation*}
for all $0<t\lesssim 1$, since $e^{-c\kappa^{-1}}\ge e^{-c} \gtrsim 1$, which concludes our theorem.
\end{proof}

We are now ready to prove \Cref{thm:2-instability} as well. 

\begin{proof}[Proof of \Cref{thm:2-instability}]
We choose $X=L^{2}_{\ol{B}_{1}}$, $Y = {\rm HS}(\mS^{n-1})$, $T = F_{\kappa}$ and $(\phi_{j})$ as the Dirichlet eigenfunctions of $-\Delta$ in $B_{1}$. For each $f \in H^1_{\ol{B}_1}$, we have 
\begin{equation*}
\norm{f}_{\gamma_{0},(\phi_{j})} \sim \norm{f}_{H^1_{\ol{B}_1}} \quad \text{with} \quad \gamma_{0} = \frac{1}{n}.
\end{equation*}
From \Cref{rem:thm2}, we have 
\begin{equation*}
\sigma_{j}(F_{\kappa}:L^{2}_{\ol{B}_{1}}\rightarrow {\rm HS}) \lesssim \min\left\{1,\kappa^{\alpha(n)}\exp\left(-c\kappa^{-1/2}j^{1/(2n)}\right)\right\}
\end{equation*}
which verifies \eqref{eq:singular-value} with 
\begin{equation*}
h_{1} \sim 1 ,\quad h_{2} \sim \kappa^{\alpha(n)} ,\quad \mu = c\kappa^{-1/2} ,\quad \beta = \frac{1}{2n}. 
\end{equation*}
By using \Cref{lem:refine-KRS}, we then see that 
\begin{equation*}
\omega(t) \gtrsim \max \left\{ t , c\kappa^{-1}(1 + \alpha(n)\log\kappa + \log(1/t))^{-2} \right\}
\end{equation*}
for all $0<t\lesssim 1$, since $e^{-c\kappa^{-1/2}}\ge e^{-c} \gtrsim 1$. This concludes our theorem. 
\end{proof}

\appendix
\section{\label{sec:numerical}Numerical simulations for singular values}

In the section we give some numerical evidence for the singular value estimates in \Cref{thm:1} and \Cref{thm:2}. We note that singular value computations related to \Cref{thm:1} are already given in \cite{GriesmaierSylvester2017, GriesmaierSylvester2017_3D} for $n=2,3$.

\addtocontents{toc}{\SkipTocEntry}
\subsection{Unique continuation for Herglotz waves}

In \Cref{rem:Jacobi-Anger} we see that the singular values $\sigma_{j}$ for the Herglotz operator $A_{\kappa}$ are given by 
\begin{equation*}
\left( \frac{(2\pi)^n}{\kappa} \int_{0}^{\kappa} r J_{\ell+\frac{n-2}{2}}(r)^{2} \, \rmd r \right)^{1/2},
\end{equation*}
with correct multiplicity and arranged in nonincreasing order. We now restrict to the case $n=3$. Using the spherical Bessel function $j_{\ell}(r) = \sqrt{\frac{\pi}{2r}} J_{\ell+\frac{1}{2}}(r)$, we obtain the following formula for the singular values (when arranged in nonincreasing order):
\begin{equation}
\sigma_{\ell}^{m}(A_{\kappa}) = 4\pi\kappa \left( \int_{0}^{1} r^{2} j_{\ell}(\kappa r)^{2} \, \rmd r \right)^{1/2} \quad \text{with multiplicity $2\ell+1$}
\end{equation}
for $\ell=0,1,2,\cdots$ and $\abs{m}\le\ell$. To illustrate the properties of the stable region, we also write $Q_{\kappa}: L^2(\mS^{n-1}) \to L^2(B_1)$, $Q_{\kappa}f = P_{\kappa} f|_{B_1} = \kappa^{-1} A_{\kappa} f$ and observe that 
\begin{equation*}
\sigma_{\ell}^{m}(Q_{\kappa}) = \kappa^{-1} \sigma_{\ell}^{m}(A_{\kappa})
\end{equation*}
when $n=3$. By the definition of $A_{\kappa}$, \Cref{thm:1} indicates that 
\begin{subequations}
\begin{align}
& \sigma_{j}(Q_{\kappa}) \sim \kappa^{-1}, && \text{for all} \quad j \lesssim \kappa^{2}, \label{eq:stable-region-Herglotz-denormalized} \\
& \sigma_{j}(Q_{\kappa}) \lesssim \kappa^{-1} \exp \left( -c \kappa^{-1} j^{1/2} \right), && \text{for all} \quad j \gtrsim \kappa^{2}, \label{eq:unstable-region-Herglotz-denormalized}
\end{align}
\end{subequations}
We plot the singular values of $A_{\kappa}$ and $Q_{\kappa}$ in \Cref{fig:Herglotz}. 

\begin{figure}[ht]
\begin{center}
\includegraphics[width=.48\textwidth]{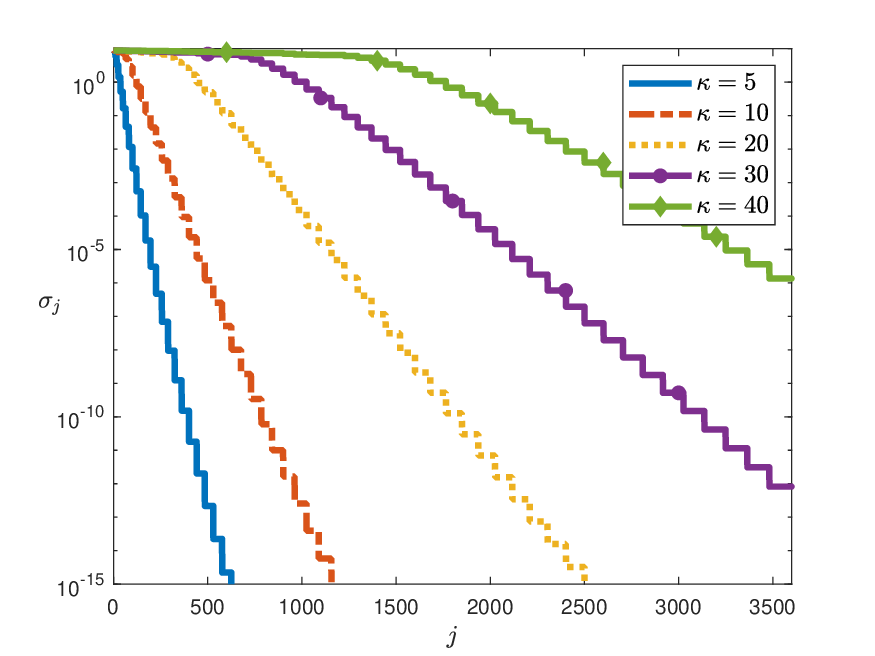} \quad 
\includegraphics[width=.48\textwidth]{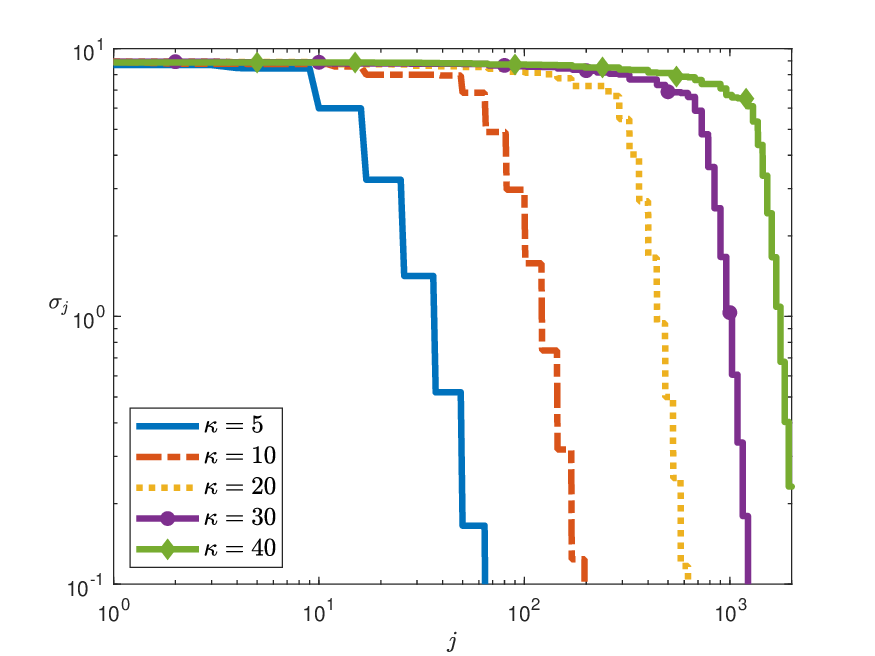}\\
\includegraphics[width=.48\textwidth]{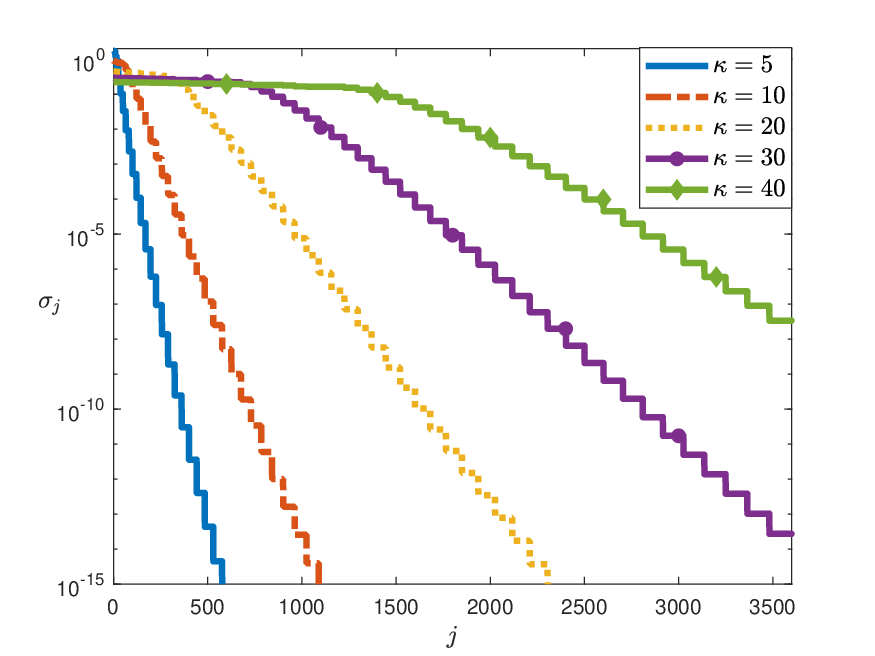} \quad \includegraphics[width=.48\textwidth]{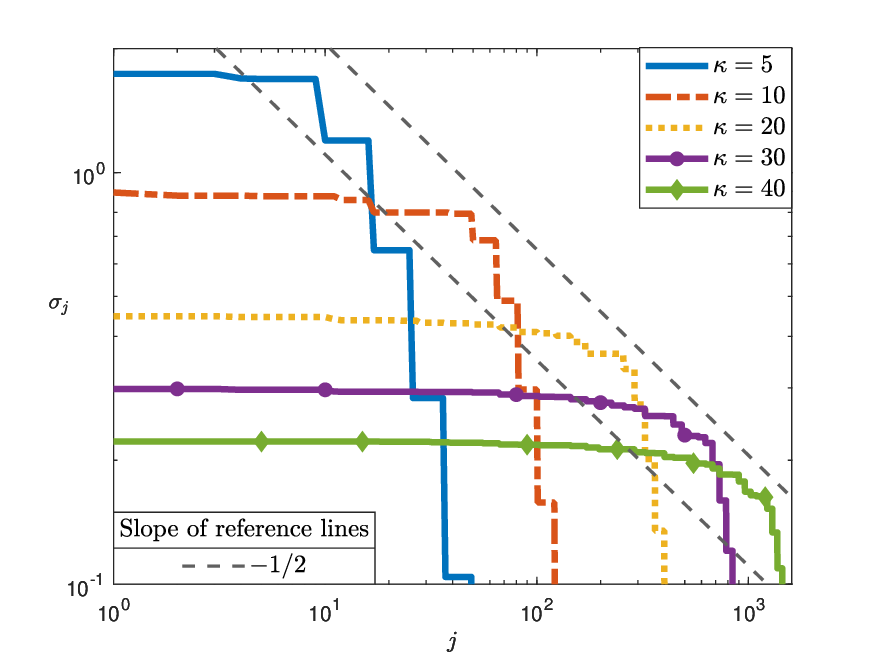} 
\end{center}
\caption{The singular values of $A_{\kappa}$ (first row) and $Q_{\kappa}$ (second row) in descending order (counting multiplicity). In the first column, the $y$-axis is of $\log$-scale, while in the second column, both $x$ and $y$-axis are of $\log$-scale.}
\label{fig:Herglotz}
\end{figure}

In \Cref{fig:Herglotz} one can distinguish the two different regions, i.e.\ stable and unstable regions of singular values. Furthermore the above numerical example verifies \Cref{thm:1} quantitatively with the following perspectives:
\begin{itemize}
\item  \Cref {thm:1} suggests that the singular values for $A_{\kappa}$ in the stable region are roughly constant (independent of $\kappa$), which is shown in the $(1,2)$ subfigure.
\item  Equations \eqref{eq:stable-region-Herglotz-denormalized} and \eqref{eq:unstable-region-Herglotz-denormalized} suggest that the `shift points' of stable and unstable regions of the singular values for $Q_{\kappa}$ are roughly $(\kappa^{2}, \kappa^{-1})$, which means the shift points lie approximately in between the lines of slope $-1/2$ for large $\kappa$. This is shown in the $(2,2)$ subfigure. 
\end{itemize}

\addtocontents{toc}{\SkipTocEntry}
\subsection{Linearized inverse scattering problem}

In this subsection we give numerical evidence for \Cref{thm:2}. For the sake of computability we work in the setting of \Cref{rem:thm2} and replace the domain $B_1$ by $\Omega=[0,1]^n$ when $n=2,3$, with $L^2(\Omega)$ identified with $L^2_{\ol{\Omega}}$. Recall that $J_{\nu}(x)$ is the Bessel function of the first kind of order $\nu$. Let $F_{\kappa}^{*}: {\rm HS} \rightarrow L^2(\Omega)$ be the adjoint operator of $F_{\kappa} : L^{2}(\Omega) \rightarrow {\rm HS}$.  By using the Schwartz kernel \eqref{kkappa_kernel}, properties of the Hilbert-Schmidt norm and the Fourier transform of the spherical surface measure, one sees that 
\begin{equation*}
\begin{aligned}
& (F_{\kappa}^{*} F_{\kappa} (h_{1}),h_{2} )_{L^2(\Omega)} = \left(F_{\kappa} (h_{1}), F_{\kappa}(h_{2})\right)_{{\rm HS}\,\left(L^{2}(\mathbb{S}^{n-1})\right)} \\
& \quad = \kappa^{n-1}\int_{\mathbb{S}^{n-1}}\int_{\mathbb{S}^{n-1}} \hat{h}_1(\kappa(\hat{x}-\hat{z})) \ol{\hat{h}_2(\kappa(\hat{x}-\hat{z}))}\,\rmd S(\hat{z}) \,\rmd S(\hat{x}) \\
& \quad = \kappa^{n-1}\int_{\Omega} \left(\int_{\Omega} \left(\int_{\mathbb{S}^{n-1}}\int_{\mathbb{S}^{n-1}}e^{\bfi\kappa(\hat{z}-\hat{x})\cdot (y-x)}\,\rmd S(\hat{z}) \,\rmd S(\hat{x}) \right) h_1(y) \, \rmd y\right) \overline{h_2(x)} \, \rmd x \\
& \quad = (2\pi)^n \kappa^{n-1} \int_{\Omega} \left(\int_{\Omega} \left|\kappa(x-y)\right|^{2-n} J_{\frac{n}{2}-1}(\kappa |x-y|)^2 h_1(y)  \,\rmd y \right) \overline{h_2(x)} \, \rmd x,
\end{aligned} 
\end{equation*}
for all $h_{1},h_{2}\in L^{2}(\Omega)$. This means that $F_{\kappa}^{*} F_{\kappa}: L^2(\Omega)\rightarrow L^2(\Omega)$ is an convolution operator, and each singular value of $F_{\kappa}$ is equal to the square root of the corresponding eigenvalue of $F_{\kappa}^{*} F_{\kappa}$. 

When $n=2$, we divide the domain $\Omega = [0,1]^{2}$ into $100 \times 100$ identical subdomains $\{\Omega_{i_{1}i_{2}}\}$, and in the spirit of Riemann integral, $F^{*}_{\kappa}F_{\kappa}$ is approximated by the equation 
\begin{equation}
\left( (F_{\kappa}^{*} F_{\kappa})^{\rm approx} \right) (h) = (2\pi)^{2} \kappa\sum_{i_{1},i_{2}=1}^{100} J_{0} (\kappa |x-y_{i_{1}i_{2}}|)^2 h(y_{i_{1}i_{2}}) |\Omega_{i_{1}i_{2}}|,
\end{equation}
where $y_{i_{1}i_{2}} \in \Omega_{i_{1}i_{2}}$. Therefore the singular values of $F_{\kappa} : L^{2}(\Omega) \rightarrow {\rm HS}$ are approximated by $\hat{\sigma}_{j} = \sqrt{|\hat{\lambda}_{j}|}$, where $\hat{\lambda}_{j}$ are the eigenvalues of the matrix  
\begin{equation*}
\left( \left( (F_{\kappa}^{*} F_{\kappa})^{\rm approx} \right) (x_{i_{1}i_{2}}) \right)_{i_{1},i_{2}=1}^{100}.
\end{equation*}

When $n=3$ we approximate the singular values in a similar way by taking  $\hat{\lambda}_{j}$ to be the eigenvalues of the matrix  
\begin{equation*}
\left( \left( (F_{\kappa}^{*} F_{\kappa})^{\rm approx} \right) (x_{i_{1}i_{2}i_{3}}) \right)_{i_{1},i_{2},i_{3}=1}^{40}.
\end{equation*}
Note that for the sake of computability we take the number of subdomains to be $40 \times 40$. Furthermore, to illustrate the properties of the stable region, we also calculate the singular values of the far-field operator $\tilde{F}_{\kappa} \equiv \kappa^{-\frac{n-1}{2}}F_{\kappa}$. From \Cref{rem:thm2} we can derive that
\begin{subequations}
\begin{align}
& \kappa^{-\frac{n-1}{2}}j^{-\frac{1}{2n}} \lesssim \sigma_{j}(\tilde{F}_{\kappa} : L^{2}_{\ol{B}_{1}} \rightarrow {\rm HS}) \lesssim \kappa^{-\frac{n-1}{2}}, && \text{for all} \quad j \lesssim \kappa^{n}, \label{eq:stable-region-farfield-L2-unnormal}\\
& \sigma_{j}(\tilde{F}_{\kappa} : L^{2}_{\ol{B}_{1}} \rightarrow {\rm HS}) \lesssim \kappa^{\alpha(n)-\frac{n-1}{2}} \exp\left(-c\kappa^{-\frac{1}{2}}j^{\frac{1}{2n}}\right), && \text{for all} \quad j \gtrsim \kappa^{n}. 
\label{eq:unstable-region-farfield-L2-unnormal} 
\end{align}    
\end{subequations}
The approximated singular values under different wave numbers $\kappa$ are exhibited in \Cref{fig:far-field}.

\begin{figure}[ht]
\begin{center}
\includegraphics[width=.48\textwidth]{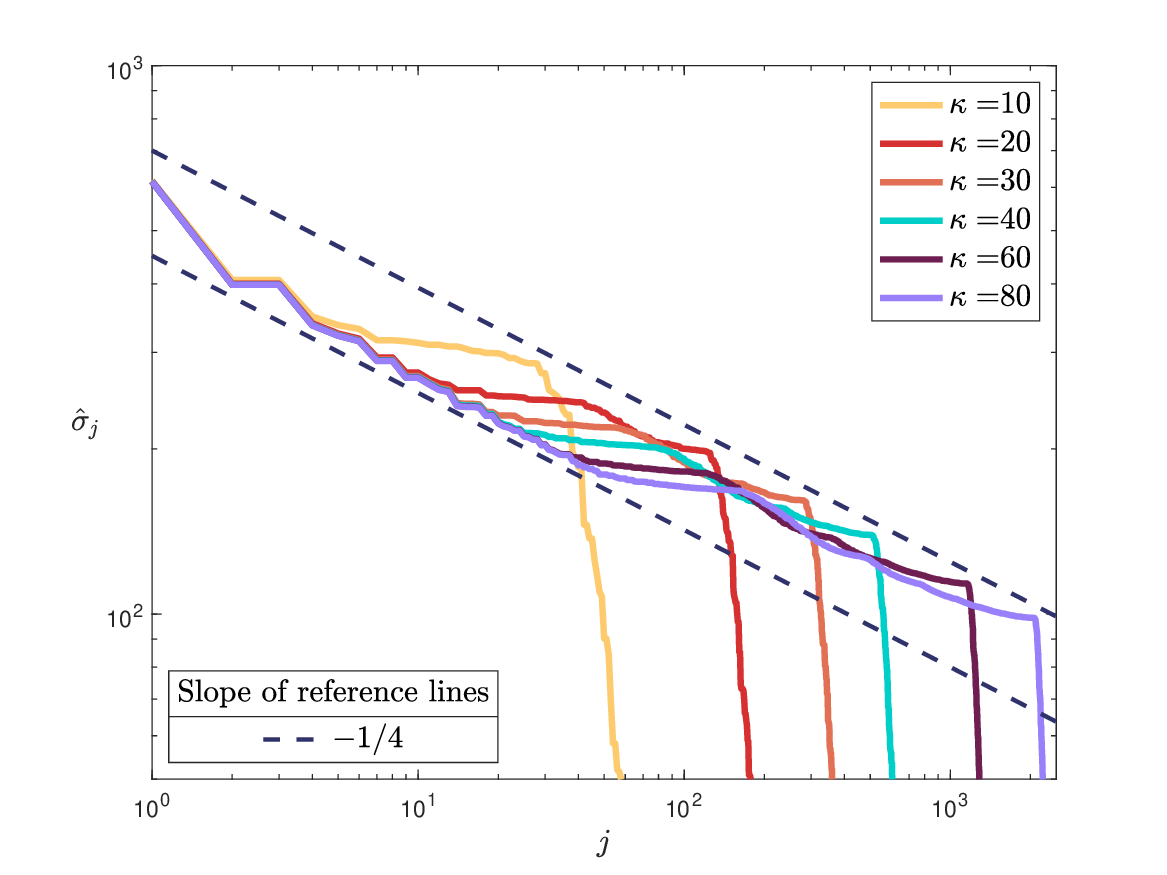} \quad \includegraphics[width=.48\textwidth]{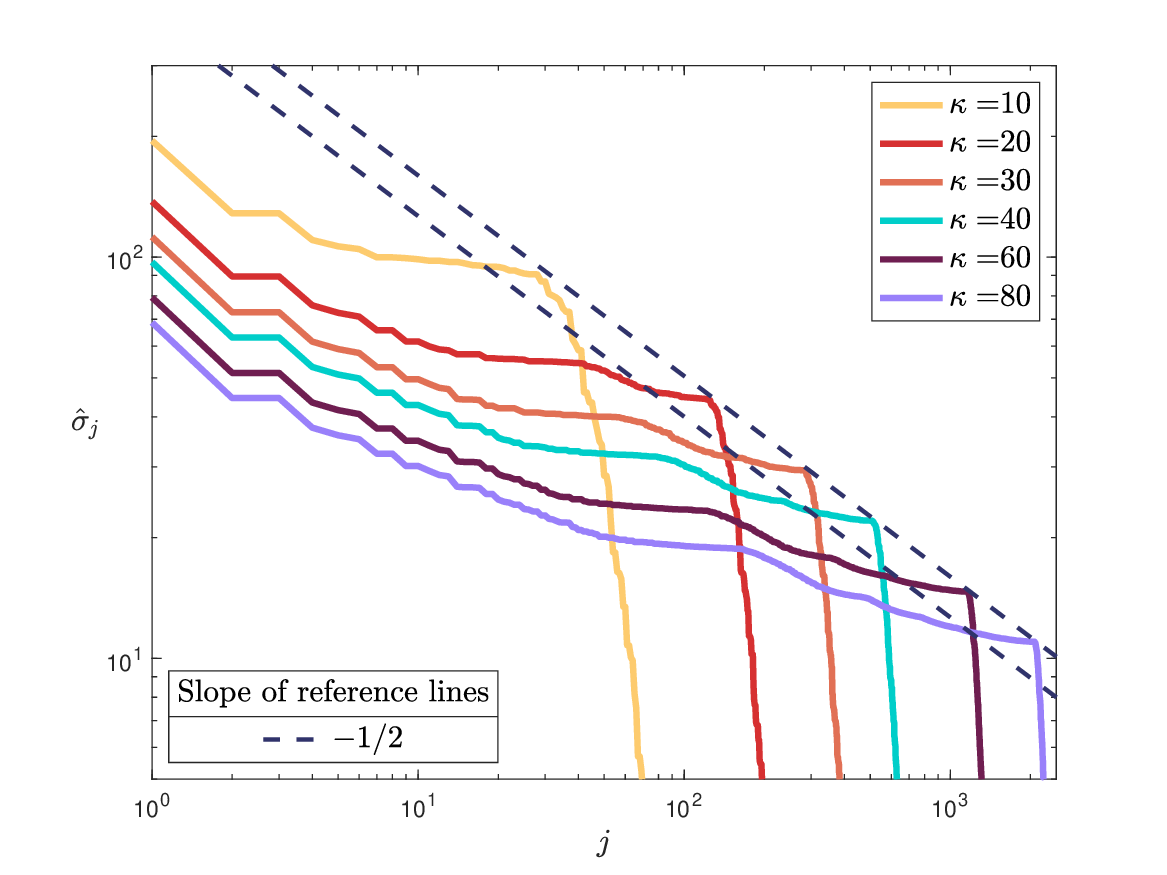}\\
\includegraphics[width=.48\textwidth]{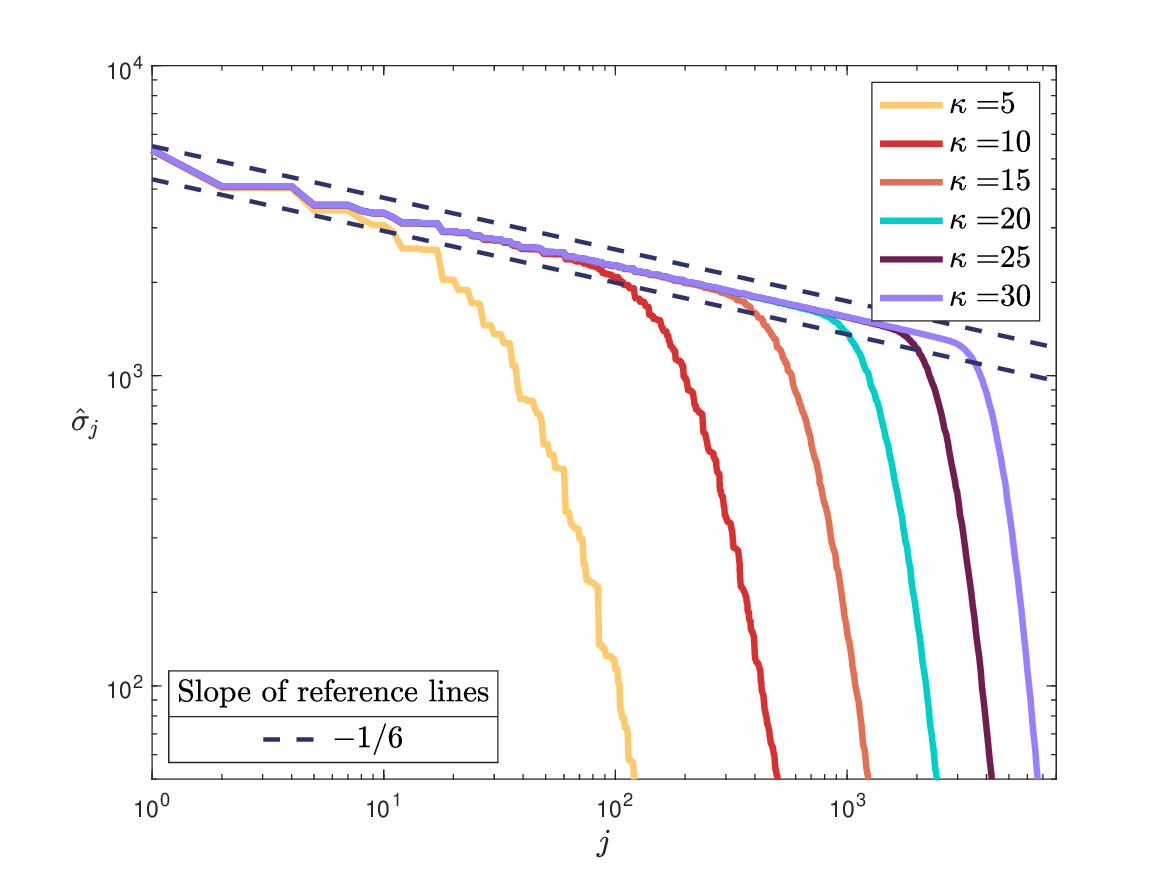} \quad \includegraphics[width=.48\textwidth]{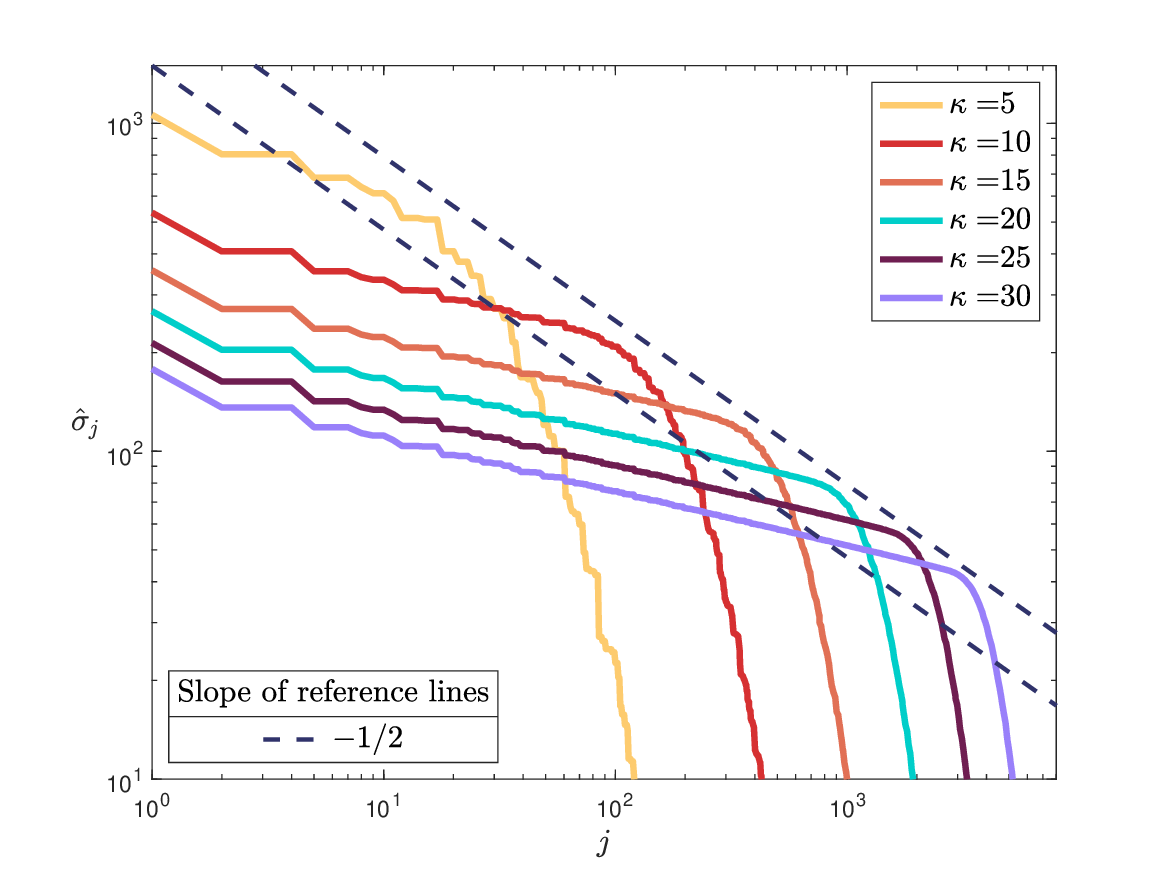}
\end{center}
\caption{Plot of singular values $\hat{\sigma}_{j}$ when $n=2$ (first row) and $n=3$ (second row).  The singular values for normalized far-field operator $F_{\kappa}$ are given in the first column while operators $\tilde{F}_{\kappa}$ are given in the second column. In each figure, both $x$ and $y$-axis are of $\log$-scale.}
\label{fig:far-field}
\end{figure}

\begin{figure}[ht]
\begin{center}
\includegraphics[width=.48\textwidth]{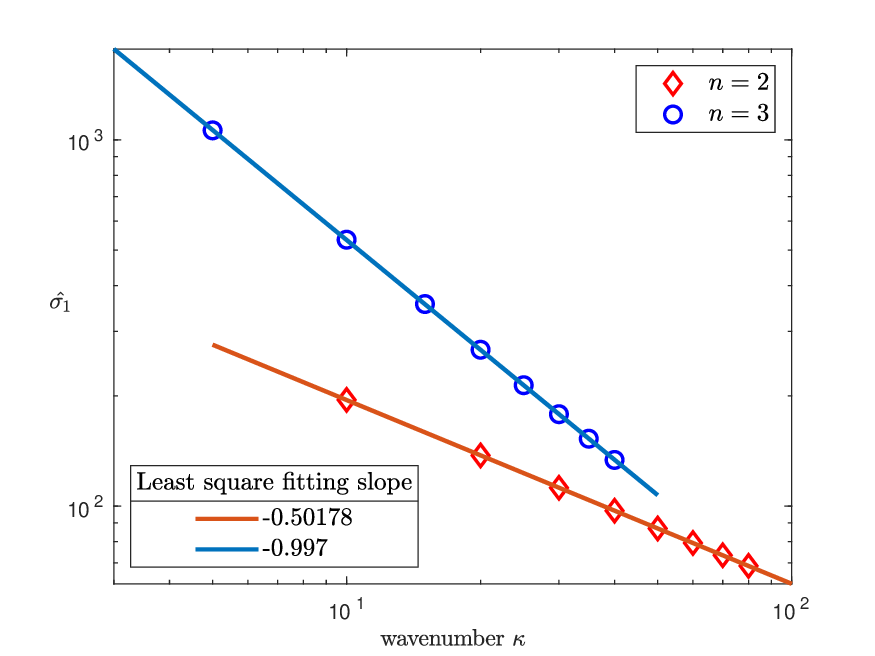} 
\end{center}
\caption{Plot of $\hat{\sigma}_{1}$ versus wave number $\kappa$}
\label{fig:far-field-fitting}
\end{figure}

The numerical results verify \Cref{rem:thm2} quantitatively with the following perspectives:  
\begin{itemize}
\item The dotted reference lines in the first column of \Cref{fig:far-field} have slope $-\frac{1}{2n}$, which agree with the asymptotic behavior of the singular values in stable region (corresponding to the term $j^{-\frac{1}{2n}}$ in the lower bound of \eqref{eq:stable-region-farfield-L2}, independent of wavenumber $\kappa$). 
\item One also sees that the the singular values start to behave differently near $j = \kappa^{n}$, which corresponds to the point $\left( \log j , \log (\kappa^{\frac{1-n}{2}}j^{-\frac{1}{2n}})  \right) = \left(  n \log \kappa , -\frac{n}{2} \log \kappa \right)$ on the log-log plot. We marked them using dotted reference lines with slope $-\frac{1}{2}$ in the second column of \Cref{fig:far-field}. 
\item When $j=1$, we have the bound $\sigma_{1} \gtrsim \kappa^{\frac{1-n}{2}}$ for the operator $\tilde{F}_{\kappa}$. Therefore, in the log-log plot, we should have the straight line parametrized by $\left( \log \kappa , \log (\sigma_{1}) \right) = \left( \log \kappa , \frac{1-n}{2} \log \kappa \right)$. The slope of such line (in the log-log plot) is $\frac{\log (\sigma_{1})}{\log \kappa} = \frac{1-n}{2}$. By using least squares fitting, we obtain a line with slope $-0.50178$ for $n=2$, and with slope $-0.99701$ for $n=3$, see \Cref{fig:far-field-fitting}, which meet our expectation quite well. 
\end{itemize}

\end{sloppypar}


\begin{thebibliography}{INUW14}
	
	\bibitem[AH76]{AgmonHormander}
	S.~Agmon and L.~H{\"{o}}rmander.
	\newblock Asymptotic properties of solutions of differential equations with
	simple characteristics.
	\newblock {\em J. Analyse Math.}, 30:1--38, 1976.
	\newblock
	\href{https://mathscinet.ams.org/mathscinet-getitem?mr=466902}{MR0466902},
	\href{https://zbmath.org/0335.35013}{Zbl:0335.35013},
	\href{https://doi.org/10.1007/BF02786703}{doi:10.1007/BF02786703}.
	
	\bibitem[BLT10]{BaoLinTriki2010}
	G.~Bao, J.~Lin, and F.~Triki.
	\newblock A multi-frequency inverse source problem.
	\newblock {\em J. Differential Equations}, 249(12):3443--3465, 2010.
	\newblock
	\href{https://mathscinet.ams.org/mathscinet/article?mr=2737437}{MR2737437},
	\href{https://zbmath.org/1205.35336}{Zbl:1205.35336},
	\href{https://doi.org/10.1016/j.jde.2010.08.013}{doi:10.1016/j.jde.2010.08.013}.
	
	\bibitem[BJK21]{Bonami}
	A.~Bonami, P.~Jaming, and A.~Karoui.
	\newblock Non-asymptotic behavior of the spectrum of the sinc-kernel operator
	and related applications.
	\newblock {\em J. Math. Phys.}, 62(3), 2021.
	\newblock Paper No. 033511, 20 pages.
	\href{https://mathscinet.ams.org/mathscinet/article?mr=4234677}{MR4234677},
	\href{https://zbmath.org/1461.81040}{Zbl:1461.81040},
	\href{https://doi.org/10.1063/1.5140496}{doi:10.1063/1.5140496},
	\href{https://arxiv.org/abs/1804.01257}{\texttt{arXiv:1804.01257}}.
	
	\bibitem[Cha06]{Chavel06RiemannianGeometry}
	I.~Chavel.
	\newblock {\em Riemannian geometry: {A} modern introduction}, volume~98 of {\em
		Cambridge Studies in Advanced Mathematics}.
	\newblock Cambridge University Press, Cambridge, second edition, 2006.
	\newblock
	\href{https://mathscinet.ams.org/mathscinet-getitem?mr=2229062}{MR2229062},
	\href{https://zbmath.org/1099.53001}{Zbl:1099.53001},
	\href{https://doi.org/10.1017/CBO9780511616822}{doi:10.1017/CBO9780511616822}.
	
	\bibitem[EKN89]{EKN89TikhonovRegularization}
	H.~W. Engl, K.~Kunisch, and A.~Neubauer.
	\newblock Convergence rates for {T}ikhonov regularisation of nonlinear
	ill-posed problems.
	\newblock {\em Inverse Problems}, 5(4):523--540, 1989.
	\newblock
	\href{https://mathscinet.ams.org/mathscinet-getitem?mr=1009037}{MR1009037},
	\href{https://zbmath.org/0695.65037}{Zbl:0695.65037},
	\href{https://doi.org/10.1088/0266-5611/5/4/007}{doi:10.1088/0266-5611/5/4/007}.
	
	\bibitem[GFRZ22]{GarciaFerreroRulandZaton}
	M.~{\'{A}}. Garc{\'{i}}a-Ferrero, A.~R{\"{u}}land, and W.~Zato{\'{n}}.
	\newblock Runge approximation and stability improvement for a partial data
	{C}alder{\'{o}}n problem for the acoustic {H}elmholtz equation.
	\newblock {\em Inverse Probl. Imaging}, 16(1):251--281, 2022.
	\newblock
	\href{https://mathscinet.ams.org/mathscinet-getitem?mr=4369055}{MR4369055},
	\href{https://zbmath.org/1481.35399}{Zbl:1481.35399},\href{https://doi.org/10.3934/ipi.2021049}{doi:10.3934/ipi.2021049},
	\href{https://arxiv.org/abs/2101.04089}{\texttt{arXiv:2101.04089}}.
	
	\bibitem[GH22]{garde2024linearised}
	H.~Garde and N.~Hyv{\"{o}}nen.
	\newblock Linearised {C}alder\'on problem: Reconstruction and {L}ipschitz
	stability for infinite-dimensional spaces of unbounded perturbations.
	\newblock {\em arXiv preprint}, 2022.
	\newblock \href{https://arxiv.org/abs/2204.10164}{\texttt{arXiv:2204.10164}}.
	
	\bibitem[GHS14]{GriesmaierHankeSylvester2014}
	R.~Griesmaier, M.~Hanke, and J.~Sylvester.
	\newblock Far field splitting for the {H}elmholtz equation.
	\newblock {\em SIAM J. Numer. Anal.}, 52(1):343--362, 2014.
	\newblock
	\href{https://mathscinet.ams.org/mathscinet/article?mr=3163247}{MR3163247},
	\href{https://zbmath.org/1295.35381}{Zbl:1295.35381},
	\href{https://doi.org/10.1137/120891381}{doi:10.1137/120891381}.
	
	\bibitem[GS17a]{GriesmaierSylvester2017}
	R.~Griesmaier and J.~Sylvester.
	\newblock Uncertainty principles for inverse source problems, far field
	splitting, and data completion.
	\newblock {\em SIAM J. Appl. Math.}, 77(1):154--180, 2017.
	\newblock
	\href{https://mathscinet.ams.org/mathscinet/article?mr=3600374}{MR3600374},
	\href{https://zbmath.org/1366.35239}{Zbl:1366.35239},
	\href{https://doi.org/10.1137/16M1086157}{doi:10.1137/16M1086157},
	\href{https://arxiv.org/abs/1607.05678}{\texttt{arXiv:1607.05678}}.
	
	\bibitem[GS17b]{GriesmaierSylvester2017_3D}
	R.~Griesmaier and J.~Sylvester.
	\newblock Uncertainty principles for three-dimensional inverse source problems.
	\newblock {\em SIAM J. Appl. Math.}, 77(6):2066--2092, 2017.
	\newblock
	\href{https://mathscinet.ams.org/mathscinet/article?mr=3725290}{MR3725290},
	\href{https://zbmath.org/1380.65347}{Zbl:1380.65347},
	\href{https://doi.org/10.1137/17M111287X}{doi:10.1137/17M111287X}.
	
	\bibitem[GS18]{GriesmaierSylvester2018_Maxwell_elasticity}
	R.~Griesmaier and J.~Sylvester.
	\newblock Uncertainty principles for inverse source problems for
	electromagnetic and elastic waves.
	\newblock {\em Inverse Problems}, 34(6), 2018.
	\newblock Paper No. 065003, 37 pages.
	\href{https://mathscinet.ams.org/mathscinet/article?mr=3801238}{MR3801238},
	\href{https://zbmath.org/1452.78018}{Zbl:1452.78018},
	\href{https://doi.org/10.1088/1361-6420/aab45c}{doi:10.1088/1361-6420/aab45c}.
	
	\bibitem[H{\"{o}}r83]{HormanderVol1}
	L.~H{\"{o}}rmander.
	\newblock {\em The analysis of linear partial differential operators {I}.
		{D}istribution theory and {F}ourier analysis}, volume 256 of {\em Grundlehren
		der mathematischen Wissenschaften [Fundamental Principles of Mathematical
		Sciences]}.
	\newblock Springer-Verlag, Berlin, 1983.
	\newblock
	\href{https://mathscinet.ams.org/mathscinet-getitem?mr=717035}{MR0717035},
	\href{https://zbmath.org/1028.35001}{Zbl:1028.35001},
	\href{https://doi.org/10.1007/978-3-642-96750-4}{doi:10.1007/978-3-642-96750-4}.
	
	\bibitem[HI04]{HI04IncreasingStabilityHelmholtz}
	T.~Hrycak and V.~Isakov.
	\newblock Increased stability in the continuation of solutions to the
	{H}elmholtz equation.
	\newblock {\em Inverse Problems}, 20(3):697--712, 2004.
	\newblock
	\href{https://mathscinet.ams.org/mathscinet-getitem?mr=2067496}{MR2067496},
	\href{https://zbmath.org/1086.35080}{Zbl:1086.35080},
	\href{https://doi.org/10.1088/0266-5611/20/3/004}{doi:10.1088/0266-5611/20/3/004}.
	
	\bibitem[IN12]{IN12EnergyRegularitymultidimensionGelfand}
	M.~I. Isaev and R.~G. Novikov.
	\newblock Energy and regularity dependent stability estimates for the
	{G}el'fand inverse problem in multidimensions.
	\newblock {\em J. Inverse Ill-Posed Probl.}, 20(3):313--325, 2012.
	\newblock
	\href{https://mathscinet.ams.org/mathscinet/article?mr=2984491}{MR2984491},
	\href{https://zbmath.org/1279.35106}{Zbl:1279.35106},
	\href{https://doi.org/10.1515/jip-2012-0024}{doi:10.1515/jip-2012-0024}.
	
	\bibitem[Isa11]{Isa11increasingstability}
	V.~Isakov.
	\newblock Increasing stability for the {S}chr\"{o}dinger potential from the
	{D}irichlet-to-{N}eumann map.
	\newblock {\em Discrete Contin. Dyn. Syst. Ser. S}, 4(3):631--640, 2011.
	\newblock
	\href{https://mathscinet.ams.org/mathscinet-getitem?mr=2746425}{MR2746425},
	\href{https://zbmath.org/1210.35289}{Zbl:1210.35289},
	\href{http://dx.doi.org/10.3934/dcdss.2011.4.631}{doi:10.3934/dcdss.2011.4.631}.
	
	\bibitem[ILW16]{ILW16IncreasingStability}
	V.~Isakov, R.-Y. Lai, and J.-N. Wang.
	\newblock Increasing stability for the conductivity and attenuation
	coefficients.
	\newblock {\em SIAM J. Math. Anal.}, 48(1):569--594, 2016.
	\newblock
	\href{https://mathscinet.ams.org/mathscinet/article?mr=3457697}{MR3457697},
	\href{https://zbmath.org/1338.35496}{Zbl:1338.35496},
	\href{https://doi.org/10.1137/15M1019052}{doi:10.1137/15M1019052},
	\href{https://arxiv.org/abs/1505.00108}{\texttt{arXiv:1505.00108}}.
	
	\bibitem[INUW14]{INUW14increasingstability}
	V.~Isakov, S.~Nagayasu, G.~Uhlmann, and J.-N. Wang.
	\newblock Increasing stability of the inverse boundary value problem for the
	{S}chr\"{o}dinger equation.
	\newblock In {\em Inverse problems and applications}, volume 615 of {\em
		Contemp. Math.}, pages 131--141. Amer. Math. Soc., Providence, RI, 2014.
	\newblock
	\href{https://mathscinet.ams.org/mathscinet-getitem?mr=3221602}{MR3221602},
	\href{https://zbmath.org/1330.35531}{Zbl:1330.35531},
	\href{http://dx.doi.org/10.1090/conm/615}{doi:10.1090/conm/615},
	\href{https://arxiv.org/abs/1302.0940}{\texttt{arXiv:1302.0940}}.
	
	\bibitem[IW21]{IW21IncreasingStabilitySource}
	V.~Isakov and J.-N. Wang.
	\newblock Uniqueness and increasing stability in electromagnetic inverse source
	problems.
	\newblock {\em J. Differential Equations}, 283:110--135, 2021.
	\newblock
	\href{https://mathscinet.ams.org/mathscinet-getitem?mr=4223340}{MR4223340},
	\href{https://zbmath.org/1459.35350}{Zbl:1459.35350},\href{https://doi.org/10.1016/j.jde.2021.02.035}{doi:10.1016/j.jde.2021.02.035},
	\href{https://arxiv.org/abs/2007.01096}{\texttt{arXiv:2007.01096}}.
	
	\bibitem[Joh60]{John60UCP}
	F.~John.
	\newblock Continuous dependence on data for solutions of partial differential
	equations with a prescribed bound.
	\newblock {\em Comm. Pure Appl. Math.}, 13:551--585, 1960.
	\newblock
	\href{https://mathscinet.ams.org/mathscinet-getitem?mr=130456}{MR0130456},
	\href{https://zbmath.org/0097.08101}{Zbl:0097.08101},
	\href{https://doi.org/10.1002/cpa.3160130402}{doi:10.1002/cpa.3160130402}.
	
	\bibitem[Kar18]{Mirza1}
	M.~Karamehmedovi{\'c}.
	\newblock Explicit tight bounds on the stably recoverable information for the
	inverse source problem.
	\newblock {\em J. Phys. Commun.}, 2(9), 2018.
	\newblock Paper No. 095021, 14 pages.
	\href{https://doi.org/10.1088/2399-6528/aad16b}{doi:10.1088/2399-6528/aad16b},
	\href{https://arxiv.org/abs/1803.05008}{\texttt{arXiv:1803.05008}}.
	
	\bibitem[KHK20]{Mirza2}
	A.~Kirkeby, M.~T.~R. Henriksen, and M.~Karamehmedovi\'{c}.
	\newblock Stability of the inverse source problem for the {H}elmholtz equation
	in {$\mathbf{R}^3$}.
	\newblock {\em Inverse Problems}, 36(5), 2020.
	\newblock Paper No. 055007, 23 pages.
	\href{https://mathscinet.ams.org/mathscinet/article?mr=4105320}{MR4105320},
	\href{https://zbmath.org/1469.35253}{Zbl:1469.35253},
	\href{https://doi.org/10.1088/1361-6420/ab762d}{doi:10.1088/1361-6420/ab762d}.
	
	\bibitem[KRS21]{KRS21InstabilityMechanism}
	H.~Koch, A.~R\"{u}land, and M.~Salo.
	\newblock On instability mechanisms for inverse problems.
	\newblock {\em Ars Inveniendi Analytica}, 2021.
	\newblock Paper No. 7, 93 pages.
	\href{https://mathscinet.ams.org/mathscinet-getitem?mr=4462475}{MR4462475},
	\href{https://zbmath.org/1482.35002}{Zbl:1482.35002},
	\href{https://doi.org/10.15781/c93s-pk62}{doi:10.15781/c93s-pk62},
	\href{https://arxiv.org/abs/2012.01855}{\texttt{arXiv:2012.01855}}.
	
	\bibitem[KW22]{KW22RefinedInstability}
	P.-Z. Kow and J.-N. Wang.
	\newblock Refined instability estimates for some inverse problems.
	\newblock {\em Inverse Probl. Imaging}, 16(6):1619--1642, 2022.
	\newblock Special issue dedicated to the memory of Victor Isakov.
	\href{https://mathscinet.ams.org/mathscinet-getitem?mr=4520377}{MR4520377},
	\href{https://zbmath.org/7675881}{Zbl:7675881},
	\href{http://dx.doi.org/10.3934/ipi.2022017}{doi:10.3934/ipi.2022017}.
	
	\bibitem[KW24]{KW24IncreasingStabilityBayesian}
	P.-Z. Kow and J.-N. Wang.
	\newblock Increasing stability in an inverse boundary value problem --
	{B}ayesian viewpoint.
	\newblock {\em preprint}, 2024.
	\newblock
	\href{http://www.math.ntu.edu.tw/~jnwang/pub/resources/ver4_Bayesian(contraction).pdf}{http://www.math.ntu.edu.tw}.
	
	\bibitem[Mel95]{Mel95GeometricScattering}
	R.~B. Melrose.
	\newblock {\em Geometric scattering theory}.
	\newblock Stanford Lectures. Cambridge University Press, Cambridge, 1995.
	\newblock
	\href{https://mathscinet.ams.org/mathscinet-getitem?mr=1350074}{MR1350074},
	\href{https://zbmath.org/0849.58071}{Zbl:0849.58071}.
	
	\bibitem[NUW13]{NUW13IncreasingStabilityHelmholtz}
	S.~Nagayasu, G.~Uhlmann, and J.-N. Wang.
	\newblock Increasing stability in an inverse problem for the acoustic equation.
	\newblock {\em Inverse Problems}, 29(2), 2013.
	\newblock 025012, 11 pages.
	\href{https://mathscinet.ams.org/mathscinet-getitem?mr=3020433}{MR3020433},
	\href{https://zbmath.org/1302.65243}{Zbl:1302.65243},
	\href{https://doi.org/10.1088/0266-5611/29/2/025012}{doi:10.1088/0266-5611/29/2/025012},
	\href{https://arxiv.org/abs/1110.5145}{\texttt{arXiv:1110.5145}}.
	
	\bibitem[NOeS23]{NegroOliveira2023}
	G.~Negro and D.~Oliveira~e Silva.
	\newblock Intermittent symmetry breaking and stability of the sharp
	{A}gmon-{H}{\"{o}}rmander estimate on the sphere.
	\newblock {\em Proc. Amer. Math. Soc.}, 151(1):87--99, 2023.
	\newblock
	\href{https://mathscinet.ams.org/mathscinet/article?mr=4504610}{MR4504610},
	\href{https://zbmath.org/1511.42010}{Zbl:1511.42010},
	\href{https://doi.org/10.1090/proc/16072}{doi:10.1090/proc/16072},
	\href{https://arxiv.org/abs/2107.14273}{\texttt{arXiv:2107.14273}}.
	
	\bibitem[PSU10]{PaivarintaSaloUhlmann2010}
	L.~P{\"{a}}iv{\"{a}}rinta, M.~Salo, and G.~Uhlmann.
	\newblock Inverse scattering for the magnetic {S}chr{\"{o}}dinger operator.
	\newblock {\em J. Funct. Anal.}, 259(7):1771--1798, 2010.
	\newblock
	\href{https://mathscinet.ams.org/mathscinet/article?mr=2665410}{MR2665410},
	\href{https://zbmath.org/1254.81092}{Zbl:1254.81092},
	\href{https://doi.org/10.1016/j.jfa.2010.06.002}{doi:10.1016/j.jfa.2010.06.002},
	\href{https://arxiv.org/abs/0908.3977}{\texttt{arXiv:0908.3977}}.
	
	\bibitem[PSU23]{PSU23book}
	G.~P. Paternain, M.~Salo, and G.~Uhlmann.
	\newblock {\em Geometric inverse problems---with emphasis on two dimensions},
	volume 204 of {\em Cambridge Studies in Advanced Mathematics}.
	\newblock Cambridge University Press, Cambridge, 2023.
	\newblock With a foreword by Andr\'{a}s Vasy.
	\href{https://mathscinet.ams.org/mathscinet/article?mr=4520155}{MR4520155},
	\href{https://zbmath.org/1519.35005}{Zbl:1519.35005},
	\href{https://doi.org/10.1017/9781009039901}{doi:10.1017/9781009039901}.
	
	\bibitem[San13]{Santacesaria13IncreasingStability}
	M.~Santacesaria.
	\newblock Stability estimates for an inverse problem for the
	{S}chr{\"{o}}dinger equation at negative energy in two dimensions.
	\newblock {\em Appl. Anal.}, 92(8):1666--1681, 2013.
	\newblock
	\href{https://mathscinet.ams.org/mathscinet/article?mr=3169124}{MR3169124},
	\href{https://zbmath.org/1302.35452}{Zbl:1302.35452},
	\href{https://doi.org/10.1080/00036811.2012.698006}{doi:10.1080/00036811.2012.698006},
	\href{https://arxiv.org/abs/1204.4610}{\texttt{arXiv:1204.4610}}.
	
	\bibitem[San15]{Santacesaria15IncreasingStability2D}
	M.~Santacesaria.
	\newblock A {H}{\"{o}}lder-logarithmic stability estimate for an inverse
	problem in two dimensions.
	\newblock {\em J. Inverse Ill-Posed Probl.}, 23(1):51--73, 2015.
	\newblock
	\href{https://mathscinet.ams.org/mathscinet/article?mr=3305939}{MR3305939},
	\href{https://zbmath.org/1332.35401}{Zbl:1332.35401},
	\href{https://doi.org/10.1515/jiip-2013-0055}{doi:10.1515/jiip-2013-0055},
	\href{https://arxiv.org/abs/1306.0763}{\texttt{arXiv:1306.0763}}.
	
	\bibitem[Tay11]{Tay11PDEvol2}
	M.~E. Taylor.
	\newblock {\em Partial differential equations {I}{I}. {Q}ualitative studies of
		linear equations}, volume 116 of {\em Applied Mathematical Sciences}.
	\newblock Springer, New York, second edition, 2011.
	\newblock
	\href{https://mathscinet.ams.org/mathscinet-getitem?mr=2743652}{MR2743652},
	\href{https://zbmath.org/1206.35003}{Zbl:1206.35003},
	\href{https://doi.org/10.1007/978-1-4419-7052-7}{doi:10.1007/978-1-4419-7052-7}.
	
\end{thebibliography}
\end{document}